\documentclass[12pt,a4paper,reqno,final]{amsart} 

\usepackage[dvips]{graphics}
\usepackage{enumerate}
\usepackage[T1]{fontenc}
\usepackage[latin1]{inputenc}
\usepackage{graphicx}
\usepackage{psfrag}
\usepackage{pst-all} 
\usepackage{pstricks-add}
\usepackage{color}
\usepackage{verbatim}

\usepackage{mathrsfs}

\numberwithin{equation}{section}  
\allowdisplaybreaks[3] 
\usepackage[dvips, hmargin={25.75mm},vmargin={39mm}]{geometry} 

\newtheorem{theorem}{Theorem}[section]

\newtheorem{lemma}[theorem]{Lemma}
\newtheorem{proposition}[theorem]{Proposition}
\newtheorem{corollary}[theorem]{Corollary}
\theoremstyle{definition}
\newtheorem{definition}[theorem]{Definition}
\newtheorem{example}{Example}
\theoremstyle{remark}
\newtheorem{remark}[theorem]{Remark}

\newtheorem{question}{Question}

\DeclareMathOperator{\lcm}{lcm}

\DeclareMathOperator{\supp}{supp} %
\DeclareMathOperator{\ft}{\mathcal{F}}

\newcommand{\tran}[1][k]{T_{#1}} 
\newcommand{\dila}[1][A^j]{D_{\! #1}} 
\newcommand{\expo}[1]{\mathrm{e}^{#1}} 

\newcommand{\meas}[1]{\abs{#1}}
\newcommand{\card}[1]{\# \lvert#1\rvert}

\newcommand{\charfct}[1]{\mathbf{1}_{#1}} 
\renewcommand{\implies}{\Rightarrow}
\renewcommand{\iff}{\Leftrightarrow}
\newcommand{\eps}{\ensuremath{\varepsilon}}

\newcommand{\union}{\cup}

\newcommand*{\numbersys}[1]{\ensuremath{\mathbb{#1}}} 

\newcommand*{\R}{\numbersys{R}}
\newcommand*{\Rn}{\numbersys{R}^n}
\newcommand*{\Q}{\numbersys{Q}}

\newcommand*{\Z}{\numbersys{Z}}
\newcommand*{\Zn}{\numbersys{Z}^n}
\newcommand*{\N}{\numbersys{N}}

\newcommand{\itvcc}[2]{\ensuremath{\left[{#1},{#2}\right]}} %
\newcommand{\itvco}[2]{\ensuremath{\left[{#1},{#2}\right)}} %
\newcommand{\abs}[1]{\ensuremath{\left\lvert#1\right\rvert}}
\newcommand{\abssmall}[1]{\ensuremath{\lvert#1\rvert}}

\newcommand{\norm}[2][]{\ensuremath{\left\lVert#2\right\rVert_{#1}}}

\newcommand{\innerprod}[3][]{\ensuremath{\left\langle #2,#3\right\rangle_{\! #1}}}

\newcommand{\innerprods}[2]{\ensuremath{\langle #1,#2\rangle}}
\newcommand{\set}[1]{\ensuremath{\left\lbrace{#1}\right\rbrace}}
\newcommand{\setprop}[2]{\ensuremath{\left\lbrace{#1} : {#2}\right\rbrace}}

\usepackage{xspace}
\newcommand{\ie}{i.e.,\xspace} 
\newcommand{\eg}{e.g.,\xspace} 

\newcommand{\ifft}{if, and only if,}

\newcommand{\bn}{\mathbb{N}}
\newcommand{\br}{\mathbb{R}}
\newcommand{\brn}{{\mathbb R^n}}

\newcommand{\bz}{\mathbb{Z}}
\newcommand{\bzn}{{\mathbb Z^n}}

\newcommand{\spa}{\operatorname{span}}

\newcommand{\ov}{\overline}
\newcommand{\ch}[1]{\mathbf{1}_{#1}}

\newcommand{\ga}{\gamma}
\newcommand{\Ga}{\LG}
\newcommand{\ve}{\varepsilon}
\newcommand{\vp}{\varphi}
\newcommand{\la}{\lambda}
\newcommand{\lan}{\langle}
\newcommand{\ran}{\rangle}

\newcommand{\af}[1]{\ensuremath{\mathcal{A}(#1)}} 

\newcommand{\lat}[1]{\ensuremath\mathsf{#1}} 
 \newcommand{\LG}{\ensuremath\lat{\Gamma}}
 \newcommand{\LL}{\ensuremath\lat{\Lambda}}
\newcommand\D{\mathscr{D}} 
\newcommand\sumstyle[2]{#1: #2} 

\usepackage[hyperindex,colorlinks,citecolor=blue]{hyperref} 
\hypersetup{
pdfview={FitH}, 
pdfstartview={FitH},
pdfauthor = {Marcin Bownik and Jakob Lemvig},
pdftitle = {Oversampling of wavelet frames for real dilations},
pdfsubject = {Wavelet frame analysis},
pdfkeywords = {wavelet, affine, oversampling, real
  dilations, frames, dual frames},
pdfcreator = {LaTeX with hyperref package},
pdfproducer = {dvips + ps2pdf}}

\begin{document}

\title{Oversampling of wavelet frames for real dilations} 
\author{Marcin Bownik}
\address{Department of Mathematics, University of Oregon, Eugene,
OR 97403--1222, U.S.A.}
\email{mbownik@uoregon.edu}

\author{Jakob Lemvig}
\address{Department of Mathematics, Technical University of Denmark,
Matematiktorvet Building 303, 2800 Kgs. Lyngby, Denmark.}
\email{j.lemvig@mat.dtu.dk}

\thanks{The first author was partially supported by NSF grant DMS-0653881.}


\date{\today}

\begin{abstract}
We generalize the Second Oversampling Theorem for wavelet frames and dual wavelet frames from the setting of integer dilations to real dilations. We also study the relationship between dilation matrix oversampling of semi-orthogonal Parseval wavelet frames and the additional shift invariance gain of the core subspace.
\end{abstract}
\maketitle

\section{Introduction}
\label{sec:intro}

Oversampling of wavelet frames has been a subject of extensive study by several authors dating back to the early 1990's. The first oversampling results are due to Chui and Shi \cite{CS, CS2}, who proved that oversampling by odd factors preserves tightness of dyadic affine frames. This is now the central result of the subject known as the Second Oversampling Theorem. Its higher dimensional generalizations to integer matrix dilations were studied by Chui and Shi \cite{CS3}, Johnson \cite{MR2048403}, Laugesen \cite{La02}, Ron and Shen \cite{RS97}. In particular, these authors introduced (in several equivalent forms) the class of oversampling matrices ``relatively prime'' to a fixed dilation $A$ and they established several oversampling results for (not necessarily tight) affine frames. Dutkay and Jorgensen \cite{DJ} shed a new light on these results by showing that oversampling of orthonormal (or frame) wavelets by such matrices leads to orthonormal (or frame) super-wavelets, respectively. 

Chui and Sun \cite{CSun2} have completed the understanding of the case of integer dilations by showing that the class of ``relatively prime'' matrices is optimal for the Second Oversampling Theorem. That is, if an oversampling matrix falls out of this class, then the oversampling does not preserve tight frame property in general. However, it is possible to give a characterization of oversampling matrices preserving tightness once affine frame generators are chosen. These results are also due to
Chui and Sun \cite{CSun1, CSun2} who extended earlier results by Catal\'an \cite{Cat}.

Despite this progress, much less is known on oversampling of affine systems generated by non-integer dilations. Chui, Czaja, Maggioni, and Weiss \cite{CCMW} have proved some results on the oversampling of tight affine frames and dual affine frames generated by special classes of real dilations. 
Hern{\'a}ndez, Labate, Weiss, and Wilson \cite{HLWW} extended the Second Oversampling Theorem to (not necessarily tight) affine frames associated with rational dilations. Moreover, they also considered more general types of scale adaptive oversampling typically arising in the study of quasi-affine systems, see also \cite{B03, BLquasi, CSS, RS97}. However, these results did not attempt to cover all possible real dilations.

The goal of this paper is to extend the Second Oversampling Theorem to arbitrary real dilations. We propose yet another condition on the oversampling lattice which guarantees preservation of frame bounds of the oversampled affine system. In the case of integer dilations, our condition is easily seen to be equivalent with the previously mentioned optimal ``relative prime'' condition. Moreover, in the case of rational dilations, our result generalizes the above mentioned result in \cite{HLWW}. Since our condition is applicable for general real dilations, including non-integer classes of dilations considered in \cite{CCMW, HLWW}, it unifies and extends previous results. In particular, our oversampling result is applicable both for non-tight frames and dual affine frames. 

To achieve our goal we introduce and study new concepts in the theory of lattices involving approximate representatives of distinct cosets and approximate duals. We have built our methods from scratch since we could not find similar results in the existing literature. We believe that our results could be of independent interest. There are two key results worth mentioning here. Our first theorem shows the existence of an approximate constellation for a suitable collection of lattices. In proving this result we have adapted the notion of a constellation, which is borrowed from the coding theory as in the work of Calderbank and Sloane \cite{cas}, to the setting of approximate coset representatives. Our second result is an extension of the duality identity 
\[
(\LG \cap \LL)^*= \ov{\LG^* + \LL^*}\qquad\text{for lattices }\LL, \LG \subset \br^n.
\]
We establish an analogue of this identity for finitely generated (but not necessarily discrete) subgroups $\LG \subset\br^n$ in terms of approximate duals. 
This is shown using arguments involving uniform distribution of sequences \cite{KN}.

The remaining elements of our techniques are more standard and involve the use of almost periodic functions. This technique was pioneered by Laugesen \cite{La01, La02} in his work on translational averaging of the wavelet functional, and later extended by Hern{\'a}ndez, Labate, Weiss, and Wilson \cite{HLW, HLWW}, and the authors \cite{BLquasi}. Finally, the last section relies on a general result about shift-invariance gain of principal shift-invariant spaces. This is a higher dimensional analogue of a result due to Aldroubi, Cabrelli, Heil, Kornelson, and Molter \cite{achkm}.

The paper is organized as follows. In Section \ref{sec:appr-transv-duals} we introduce and study the notions of approximate transversals and approximate duals. 
In Section \ref{sec:oversampling} we show the generalization of the Second Oversampling Theorem to real dilations.
In Section \ref{sec:dual} we show oversampling results for dual affine
frames. We also give a counterexample to one of the results claimed in
the paper of Chui, Czaja, Maggioni, and Weiss \cite{CCMW}.  Finally, in
Section \ref{sec:gain} we show results on the equivalence of tight
frame preservation for dilation matrix oversampling of the translation
lattice and the membership in Behera--Weber classes of wavelets \cite{Be, We}.

We end this introduction by reviewing some basic definitions. A \emph{frame} for a separable
Hilbert space $\mathcal{H}$ is a countable collection of vectors $\{f_j\}_{j \in
  \mathbb{J}}$ for which there are
constants $0 < C_1 \leq C_2 < \infty$ such that
\[ 
C_1 \norm{f}^2 \leq \sum_{j \in \mathbb{J}} \abs{\innerprod{f}{f_j}}^2 \leq
C_2 \norm{f}^2 \qquad\text{for all }f\in \mathcal{H}.
\]
If the upper bound in the above inequality holds, then $\{f_j\}$ is
said to be a \emph{Bessel sequence} with Bessel constant $C_2$.   
A frame $\{f_j\}$ is said to be \emph{tight} if we can
 choose $C_1 = C_2$; if, furthermore, $C_1= C_2=1$, then the
 sequence $\{f_j\}$ is said to be a \emph{Parseval frame}.

 Two Bessel sequences $\{f_j\}$ and $\{g_j\}$ are said to be \emph{dual 
 frames} if  
 \[ f = \sum_{j \in \mathbb{J}} \innerprod{f}{g_j}f_j \qquad \text{for all } f
 \in \mathcal{H}. 
 \] 
 It can be shown that two such Bessel sequences indeed are frames, and
 we shall say that the frame $\{g_j\}$ is \emph{dual} to $\{f_j\}$, and
 vice versa. The book by Christensen \cite{oc_03} serves as an introduction to the frame theory.

For $f \in L^1(\Rn)$, the Fourier transform is defined by 
\[
\ft
f(\xi)=\hat f(\xi) = \int_{\Rn} f(x)\mathrm{e}^{-2 \pi i
  \innerprod{\xi}{x}} \mathrm{d}x
  \]
with the usual extension to
$L^2(\Rn)$. We will frequently prove our results on the following dense subspace
of $L^2(\Rn)$
\begin{equation}
\label{eq:def-D}
\D  = \setprop{f \in L^2 (\Rn)}{\hat f \in L^\infty(\Rn) \text{ and }
\supp \hat f \text{ is compact in } \Rn \setminus \set{0} }.
\end{equation}

\section{Approximate transversals and duals}
\label{sec:appr-transv-duals}

In this section we introduce some new notions in the theory of lattices that are understood here as discrete subgroups of $\Rn$. We refer to the book by Cassels \cite{MR1434478} for basic properties of lattices. In particular, we introduce and study the notions of: approximate representatives of distinct cosets, an approximate transversal constellation, and an approximate dual. We shall build our theory from scratch since we could not find such results in the existing literature.

The notion of a constellation is frequently used in the coding theory. In particular, Calderbank and Sloane \cite{cas} have investigated signal constellations consisting of a finite number of points from a lattice $\LL$, with an equal number of points from each coset of a sublattice $\LG \subset \LL$. We shall use the same notion albeit in the approximate sense defined below.

\begin{definition}
Suppose that $\LG \subset \LL$ are two full rank lattices in $\Rn$ and
$\ve\ge0$. We say that a set $D=\{d_1,\ldots, d_l\}$, where $l =
\card{\LL/\LG}$ is the order of the quotient group, is an
$\ve$-approximate transversal of $\LL/\LG$ if there exists set
$D'=\{d_1',\ldots, d_l'\}$ of representatives of distinct cosets of
$\LL/\LG$ such that $|d_i-d_i'|\le\ve$ for all $i=1,\ldots,l$. We say
that a multiset (set with multiplicities) $K$ is an $\ve$-approximate transversal constellation if it is a union of a finite number of $\ve$-approximate transversals.
\end{definition}

The following result is a generalization of \cite[Lemma
23.19]{MR0156915}, see also \cite[Lemma 3.6]{BLquasi}.

\begin{lemma}\label{apple}
Suppose that $\LG \subset \LL$ are two full rank lattices in $\Rn$. Suppose that $D$ is an $\ve$-approximate transversal constellation of $\LL/\LG$ for some $\ve\ge 0$. Then
\begin{equation*}
\frac{1}{\card{D}}  \sum_{d \in D} \expo{2\pi i \innerprod{m}{d}} =
  \begin{cases}
    1 +  O(\abs{m}\ve) & m \in \LL^\ast ,\\
    O(\abs{m}\ve) & m \in \LG^* \setminus \LL^\ast,
  \end{cases}
\end{equation*}
as $\eps \to 0$.
\end{lemma}

\begin{proof}
Without loss of generality, we can assume that $D=\{d_1,\ldots, d_l\}$ be an 
$\ve$-approximate transversal of $\LL/\LG$, \ie we assume that $D$ is a constellation of only one transversal.  Let $D'=\{d_1',\ldots, d_l'\}$ 
be representatives of distinct cosets
of $\LL/\LG$. Then, for any $m \in \Rn \supset \LG^\ast$,
\begin{align*}
  \abs{\sum_{i=1}^l \expo{2\pi i \innerprod{m}{d_i}} -  \sum_{i=1}^l
    \expo{2\pi i \innerprods{m}{d_i'}}  } &\le \sum_{i=1}^l \abs{ \expo{2\pi i \innerprod{m}{d_i}} - 
    \expo{2\pi i \innerprods{m}{d_i'}} } \\
 &\le \sum_{i=1}^l 2\pi \abs{ \innerprod{m}{d_i} - \innerprods{m}{d_i'} }\\
  &\le \sum_{i=1}^l 2\pi \abs{m} \abs{d_i  - d_i'} \le 2\pi
  \abs{m} l \, \eps.
\end{align*}
 Since, by Hewitt and Ross   \cite[Lemma
23.19]{MR0156915}, 
 we have
\begin{equation*}
\frac{1}{l}  \sum_{i=1}^l \expo{2\pi i \innerprod{m}{d_i'}} =
  \begin{cases}
    1 & m \in \LL^\ast ,\\
    0 & m \in \LG^* \setminus \LL^\ast,
  \end{cases}
\end{equation*}
the lemma is proved. 
\end{proof}

We shall need a result about the existence of an approximate transversal constellation for suitable collections of lattices which is of independent interest. A prototype of this result for exact coset representatives takes the following form.

\begin{lemma}\label{scr}
Suppose that we have a finite number of full rank lattices $\LG_i
\subset \LL_i$, $i=1,\ldots, J$, and let $l_i = \card{ \LL_i /
  \LG_i}$. Assume that there exists a full rank lattice $\Delta$ such that
\begin{align}
\label{scr1}
\Delta & \subset \bigcap_{i=1}^J \LL_i ,
\\
\label{scr2}
\LL_i &\subset \Delta + \Ga_i  \qquad\text{for each }i=1,\ldots, J.
\end{align} 
Then, there exists a finite multiset $K \subset \Delta$ such that
\begin{align*}
\frac{\card{K \cap (\gamma + \LG_i)}}{\card{ K}} &= \frac{1}{l_i}  \qquad\text{for all }\ga \in \LL_i, \ i=1,\ldots, J.
\end{align*} 
In other words, $K$ consists of an equal number of points from each coset of $\LL_i / \LG_i$, simultaneously for all $i=1,\ldots, J$.
\end{lemma}

In general, the condition \eqref{scr1} is too restrictive for our purposes since the intersection $\bigcap_{i=1}^J \LL_i $ might be trivial and then \eqref{scr2} cannot hold. To remedy this situation, we shall need a variant of Lemma \ref{scr} for approximate coset representatives. We shall skip the proof of Lemma \ref{scr} since it follows by a direct modification of the proof of Lemma \ref{sacr}.

\begin{lemma}\label{sacr}
Suppose that we have a finite number of full rank lattices $\LG_i \subset \LL_i$, $i=1,\ldots, J$, and let $l_i = \card{ \LL_i / \LG_i}$. Assume that, for all $\ve>0$, we have\begin{align}
\label{sacr1}
\LL_i &\subset \Delta(\ve) + \Ga_i  \qquad\text{for each }i=1,\ldots, J, \text{ where }
\\
\label{sacr2}
 \Delta(\ve) & := \bigcap_{i=1}^J (\LL_i + B(0,\ve)).
\end{align} 
Then, for all sufficiently small $\ve>0$ there exists a finite multiset $K=K(\ve) \subset \Delta(\ve)$ such that
\begin{align}
\label{sacr4}
\frac{\card{K \cap (\gamma + \LG_i +B(0,\ve))}}{\card{ K}} &= \frac{1}{l_i}  \qquad\text{for all }\ga \in \LL_i, \ i=1,\ldots, J.
\end{align} 
\end{lemma}

Lemma \ref{sacr} can also be formulated in the language of approximate
transversals.  Suppose that, for all $\ve>0$, there exists a set $\Delta$
lying in the $\ve$-neighbourhood of each $\LL_i$ and containing
$\ve$-approximate transversals of each $\LL_i /\LG_i$. Then there
exists a subset $K \subset \Delta$ which is $\ve$-approximate
transversal constellation  simultaneously for each $\LL_i /\LG_i$.

\begin{proof}[Proof of Lemma \ref{sacr}]
For any $\ve>0$ we define $\delta=\ve/J$.
For each $i=1,\ldots, J$, by \eqref{sacr1} we can choose $\delta$-approximate transversal $D^i:=\{d_1^i, \ldots, d_{l_i}^i \} \subset \Delta(\delta)$ of $\LL_i / \LG_i$. Define the set $K$ as an algebraic sum
\[
K= D^1 + \ldots + D^J.
\]
Treating $K$ as a multiset, $K$ has exactly $\prod_{i=1}^J l_i$ elements. Moreover, since each $D^i \subset \Delta(\delta)$, we have that $K \subset \Delta(\delta J ) = \Delta(\ve)$. 

Fix some $1 \le i_0 \le J$ and consider 
\begin{equation}\label{sacr6}
k_0 = \sum_{i=1, \ i \ne i_0}^J d_{m_i}^i, \qquad\text{for some choice of } 1 \le m_i \le l_i.
\end{equation}
We claim that $\{d^{i_0}_1 + k_0, \ldots, d^{i_0}_{l_{i_0}} +k_0\}$ is a $\ve$-approximate transversal of $\LL_{i_0} / \LG_{i_0}$. 
Indeed, by  \eqref{sacr2} we can find $D'^i=\{d'^i_1, \ldots, d'^i_{l_{i}}\} \subset \LL_{i_0}$ such that
\begin{equation}\label{sacr8}
|d'^i_j - d^{i}_j| < \delta
\qquad\text{for } i=1,\ldots, J, \ j=1,\ldots, l_i.
\end{equation}
Furthermore, if $\delta>0$ is sufficiently small, then $D'^{i_0}$ is an exact transversal (a set of representatives of distinct cosets) of $\LL_{i_0} / \LG_{i_0}$. Let 
\[
k'_0 = \sum_{i=1, \ i \ne i_0}^J d'^i_{m_i} \in \LL_{i_0}.
\]
Clearly, $\{d'^{i_0}_1 + k'_0, \ldots, d'^{i_0}_{l_{i_0}} +k'_0\}$ is also an exact transversal of $\LL_{i_0} / \LG_{i_0}$. Thus, by \eqref{sacr8} $\{d^{i_0}_1 + k_0, \ldots, d^{i_0}_{l_{i_0}} +k_0\}$ is a $\delta J = \ve$-approximate transversal of $\LL_{i_0} / \LG_{i_0}$. This proves the claim.

Since there are precisely $\prod_{i=1, i \ne i_0}^J l_i$ elements $k_0$ of the form \eqref{sacr6}, $K$ is a union of the same number of  $\ve$-approximate transversals of $\LL_{i_0} / \LG_{i_0}$. Take sufficiently small $\ve>0$, say 
\[
0<\ve < \min\{ |\la|: 0 \ne \la \in \LL_{i_0} \}/2.
\]
Then we observe that $\ve$-neighbourhoods of distinct cosets of  $\LL_{i_0} / \LG_{i_0}$ are disjoint. Thus, for any $\ga \in \LL_{i_0}$, 
\[
\frac{\card{K \cap (\gamma + \LG_{i_0} +B(0,\ve))}}{\card{ K}} = \frac{\prod_{i=1, i \ne i_0}^J l_i}{\prod_{i=1}^J l_i}=  \frac{1}{l_{i_0}} .
\]
Since $1 \le i_0 \le J$ was arbitrary, this shows that \eqref{sacr4} holds for all sufficiently small $\ve>0$.
\end{proof}

We shall also need some additional results about approximate duals of finitely generated subgroups of $\Rn$ which are of independent interest. 

\begin{definition}
Suppose that $F \subset \Rn$ and $\ve \ge 0$. Define an $\ve$-approximate dual of $F$ as 
\[
F^{*,\ve} = \{ x\in \Rn: \lan x, g \ran \in \Z +[-\ve, \ve] \text{ for all } g\in F \}.
\]
In the case when $\ve=0$, we say that $F^{*,0}$ is an exact dual of $F$, which is denoted simply by $F^*$.
\end{definition}

The following basic proposition justifies the name for an approximate dual.

\begin{proposition} \label{elm}
Suppose that $\LL$ is a full rank lattice and $F$ is a basis of $\LL$. Then, for sufficiently small $\ve>0$ we have $\LL^{*,\ve}=\LL^*$. Furthermore, for every $\ve>0$ there exist $\delta,\ve'>0$ such that
\[
\LL^* + B(0,\ve') \subset F^{*,\delta} \subset \LL^* + B(0,\ve).
\]
\end{proposition}

The proof of Proposition \ref{elm} is left to the reader.

\begin{lemma}\label{weyl}
Suppose that $G$ is a finitely generated subgroup of $\Rn$ such that $G \cap \Zn = \{0\}$. Then for any finite subset $F \subset G$, $\ve>0$, and full rank sublattice $\LG \subset \Zn$ we have
\begin{equation}\label{weyl1}
F^{*,\ve} + \LG = \Rn.
\end{equation}
\end{lemma}

\begin{proof}
First we shall establish a slightly weaker conclusion
\begin{equation}\label{weyl2}
F^{*,\ve} + \Zn = \Rn.
\end{equation}
Then, we shall see that \eqref{weyl2} implies \eqref{weyl1}.

Assume first that $F =\{g_1,\dots, g_d\} \subset G$ is linearly independent over $\Q$. For any $k\in \Zn$ we define a vector $x_k = ( \lan k, g_1 \ran, \ldots, \lan k, g_d \ran ) \in \R^d$. Let $k_1,k_2, \ldots$ be an ordering of all elements of $\Zn$ such that $i  \le  j$ implies $||k_i||_\infty \le ||k_j||_\infty$. We claim that the sequence  $\{x_{k_i}\}_{i\in \N}$ of vectors in $\R^d$ is uniformly distributed (u.~d.) mod $1$. By the Weyl Criterion, see \cite[Theorem 6.2 in Chapter 1]{KN}, this is equivalent to the fact that the sequence of scalars $\{\lan h, x_{k_i} \ran \}_{i\in \N}$ is u.~d.~mod $1$ for any $0 \ne h\in \Z^d$. Observe that
\[
\lan h, x_k \ran = \lan k, y \ran, \qquad\text{where }y= \sum_{j=1}^d h_j g_j \text{ and } h=(h_1,\ldots, h_d).
\]
Moreover, by our hypothesis  $G \cap \Zn = \{0\}$, $y$ has at least one irrational coordinate. Repeating the same argument as in \cite[Lemma 3.2  in Chapter 2]{B03a} shows that the sequence $ \{ \lan k_i, y \ran\}_{i\in \N}$ is u.~d.~mod 1. This proves the claim. As a consequence \eqref{weyl2} holds.

Next, let $F=\{ g_1, \ldots, g_{d'}\} \subset G$ be an arbitrary finite subset of $G$. By rearranging the order of elements in $F$, we can assume that for some $d\le d'$, $\{g_1,\dots, g_d\} $ are  linearly independent over $\Q$,  and the rest of the elements of $F$ are linear combinations of thereof; that is, for $d<i\le d'$
\[
g_i = \sum_{j=1}^d c_j g_j, \qquad c_j \in \Q.
\]
Thus, we can find $N\in \N$ such that  for $d<i\le d'$
\begin{equation}\label{weyl6}
g_i = \sum_{j=1}^d d_j \tilde g_j, \qquad d_j \in \Z,
\end{equation}
where $\tilde g_j = g_j/N$. Let  $\tilde F=\{ \tilde g_1, \ldots, \tilde g_{d}, g_{d+1},\ldots, g_{d'} \} \subset \frac{1}{N}G$. By the already established case we have
\[
{\{\tilde g_1,\ldots, \tilde g_d\}}^{*,\ve} + \Zn = \Rn.
\]
Since $\ve>0$ is arbitrary, using \eqref{weyl6} we can deduce that
\[
{\tilde F}^{*,\ve} + \Zn = \Rn.
\]
We also observe that ${\tilde F}^{*,\ve} \subset F^{*,N\ve}$. This proves \eqref{weyl2} since $\ve>0$ is arbitrary.

Finally, let $\LG \subset \Zn$ be a full rank lattice. There exists $N\in\N$ such that $N\Zn \subset \LG$. The assumption $G \cap \Zn = \{0\}$ actually implies that $G \cap \Q^n = \{0\}$. In particular, we have $(NG)\cap \Zn =\{0\}$. Applying \eqref{weyl2} to a finite subset $NF \subset NG$ we have
\[
\Rn=(NF)^{*,\ve} + \Zn=\frac{1}{N} F^{*,\ve} + \Zn.
\]
Thus,
\[
\Rn =  F^{*,\ve} + N \Zn \subset  F^{*,\ve} + \LG,
\]
which completes the proof of Lemma \ref{weyl}.
\end{proof}

Our next result is a generalization of a duality identity for lattices. If $G$ is a lattice, then we have $(G \cap \Zn)^*=\ov{G^* +\Zn}$, see Corollary \ref{aduc}. However, this conclusion might fail if $G$ is not a discrete subgroup of $\Rn$. For example, if $G \subset \Rn$ is dense in $\Rn$ and $G\cap \Zn =\{0\}$, then this identity fails. Nevertheless, we have the following extension of this identity as a consequence of Lemma \ref{weyl}.

\begin{theorem}\label{adu}
Suppose that $G$ is a finitely generated subgroup of $\Rn$. Define a lattice (not necessarily full rank) $\LG = G \cap \Zn $. Then for any finite subset $F \subset G$ and $\ve>0$ we have
\begin{equation*} 
\LG^* \subset F^{*,\ve} + \Zn.
\end{equation*}
\end{theorem}

\begin{proof}
Define the ``rational'' subgroup of $G$ by $G_1 = G \cap \Q^n$. Since $G$ is finitely generated, $G_1 \subset \Q^n$ is a lattice (not necessarily full rank). Observe that the quotient group $G/G_1$ is torsion free and finitely generated. Hence, using the structure theorem for finitely generated abelian groups, we can find a complimentary subgroup $G_2 \subset G$ such that the group $G$ decomposes as an algebraic sum $G=G_1+G_2$ with $G_1 \cap G_2 = \{0\}$. 

Let $F=\{g_1, \ldots, g_d \} \subset G$ be any finite subset. We decompose each element of $F$ as $g_i=g^{(1)}_i + g^{(2)}_i$, where $g^{(j)}_i \in G_j$, $j=1,2$. Let $F_j=\{g^{(j)}_1,\ldots, g^{(j)}_d\}$.  Observe that $(F_1)^{*,\ve} \cap (F_2)^{*,\ve} \subset F^{*,2\ve}$. Moreover, we have $G_1^* \subset (F_1)^{*,\ve}$. Thus, it suffices to show that
\begin{equation}\label{adu6}
\LG^* \subset (G_1^* \cap (F_2)^{*,\ve}) + \Zn.
\end{equation}

Take any $x\in \LG^*$.
Since $\LG= G \cap \Zn = G_1 \cap \Zn$, by taking duals we have $\LG^*= G_1^*+ \Zn$. Thus, we can write $x=x_1+z_1$, where $x_1\in G_1^*$ and $z_1 \in \Zn$. Since $G_1^* \cap \Zn$ is a full rank lattice, by Lemma \ref{weyl} we have
\[
(F_2)^{*,\ve} + (G_1^* \cap \Zn) = \Rn.
\]
Hence, we can write $x_1=x_2+z_2$, where $x_2\in (F_2)^{*,\ve}$ and $z_2 \in G_1^* \cap \Zn$. Consequently, $x_2 \in  (F_2)^{*,\ve} \cap G_1^*$ and $x-x_2=z_1+z_2\in \Zn$. This proves \eqref{adu6} and completes the proof of Theorem \ref{adu}.
\end{proof}

As a consequence of Theorem \ref{adu} we can deduce the duality identity for lattices \eqref{aduc0}.

\begin{corollary}\label{aduc}
Suppose $\LG$ and $\LL$ are two lattices in $\Rn$. Then, 
\begin{equation}\label{aduc0}
(\LG \cap \LL)^*= \ov{\LG^* + \LL^*}.
\end{equation}
\end{corollary}

\begin{proof}
The inclusion $\supset$ follows immediately from the definition of a dual lattice to $\LG \cap \LL$. To show the converse inclusion we can assume that both $\LL$ and $\LG$ are full rank lattices. Indeed, if $\LL$ is not full rank lattice, then we can find a full rank lattice $\LL_0 \supset \LL$ such that $\LG \cap \LL_0 = \LG \cap \LL$. Assuming that \eqref{aduc0} holds for the pair $\LG$ and $\LL_0$ yields the same conclusion for $\LG$ and $\LL$. An identical argument works for $\LG$. Moreover, by the change of basis argument we can assume that $\LG=\Zn$. 

Let $F$ be a basis of $\LL$. Applying Theorem \ref{adu} with $G=\LL$ yields
\[
(\LL \cap \Zn)^* \subset F^{*,\delta} + \Zn \qquad\text{for all }\delta>0.
\]
Applying Proposition \ref{elm} we have 
\[
(\LL \cap \Zn)^* \subset \LL^* + \Zn + B(0,\ve) \qquad\text{for all }\ve>0.
\]
Thus, $(\LL \cap \Zn)^* \subset \ov{\LL^* + \Zn}$ as required. This completes the proof of Corollary \ref{aduc}. 
\end{proof}

It is worth mentioning that the other duality identity 
\[
(\LG + \LL)^* = \LG^* \cap \LL^*
\]
is much easier to prove since it follows directly from the definition of a dual lattice. Corollary \ref{aduc} can be also deduced from it and the fact that $\LG^{**}=\ov{\LG}$ for an arbitrary subgroup $\LG \subset \Rn$.

\section{Oversampling}
\label{sec:oversampling}

In this section we introduce our condition on the oversampling lattice \eqref{eq:condition-O-strong} to show the generalization of the Second Oversampling Theorem in the setting of real dilations.

\subsection{Oversampling the wavelet system}
\label{sec:oversampling-1}
Let $\Psi = \set{\psi_1, \dots, \psi_L} \subset L^2(\Rn)$, let $\LG$
be a lattice in $\Rn$, and let $A$ be a fixed $n \times n$ expansive
matrix, \ie all eigenvalues $\lambda$ of $A$ satisfy $\abs{\lambda} >
1$. The wavelet system generated by
 $\Psi$ is 
\begin{equation}\label{eq:lg-af}
\af{\Psi} = \setprop{\psi_{j,\gamma}}{j\in\Z,\ \gamma \in\LG, \psi \in \Psi},
\end{equation}
where
\begin{equation*}
\psi_{j,\gamma} := \dila[A^j] \tran[\gamma] \psi = \abs{\det A}^{j/2}
  \psi(A^j\cdot-\gamma)  \qquad \text{for $j \in \Z, \gamma \in \LG$}.
\end{equation*}
Here, $D_A f(x)= |\det A|^{1/2} f(Ax)$ is the dilation operator and $T_\ga f(x) =f(x-\gamma)$ is a translation operator.
If we need to stress the dependence of the
underlying dilation matrix $A$
and translation lattice $\LG$, we say that the wavelet system $\mathcal{A}(\Psi)$ is
associated with $(A,\LG)$, or we use the notation
$\mathcal{A}(\Psi,A,\LG)$ for \eqref{eq:lg-af}.

In our study of wavelet systems, it will not be necessary to consider arbitrary 
translation lattices $\LG$, and we shall restrict our attention
to the standard translation lattice $\Zn$. Indeed, for $A \in GL_n(\R)$
 expansive and $\LG = P \Zn$ for some $P \in GL_n(\R)$ consider the wavelet system 
$\af{\Psi, A,\LG}$. By the commutator
relations 
\[ \tran \dila[A] = \dila[A] \tran[Ak]  \quad \text{and} \quad \dila[(P^{-1}AP)^j]
\dila[P] = \dila[P] \dila[A^j], \] 
we see that
\begin{equation}
 \af{\dila[P]\Psi,\tilde A,\Zn} =
\dila[P]\left(\af{\Psi,A,\LG}\right),\label{eq:stan-gen-trans}
\end{equation}
where the matrix $\tilde A:= P^{-1}AP$ is similar to $A$. Observe
that a matrix similar to an expansive matrix is
again expansive as it has the same eigenvalues. Since $\dila[P]$ is
unitary, properties such as the frame and Bessel property carry over
between the two systems. Hence, it is possible to reduce studies of
wavelet systems with general translation lattices to the setting of 
integer lattice. An example of such a reduction technique is given in Corollary \ref{gof}.

Therefore, we can without loss of generality restrict attention
to wavelet systems associated with 
$(A,\Zn)$, \ie  
\begin{equation*}
 \mathcal{A}(\Psi) =  \setprop{\psi_{j,k}}{j \in \Z, k \in \Zn, \psi
\in \Psi},
\end{equation*}
and oversampling of such systems. 
Let $\LL$ be a lattice in $\Rn$ containing the integer
lattice $\Zn$, \ie $\Zn \subset \LL$.  Then, 
\[ \LL^\ast \subset \Zn \subset \LL, \]
where the dual lattice of $\LL$ is
\[
\LL^* = \{ \eta \in \Rn: \lan \eta, \lambda \ran  \in\Z \text{ for all }\lambda \in \LL\}.
\]
The $\LL$-oversampled wavelet
system is just a normalized version of the original wavelet system
with translation lattice $\LL$:
\begin{equation*}
\af{d(\LL)^{1/2}\Psi,A,\LL} = \setprop{ d(\LL)^{1/2} \, \psi_{j,\lambda}}{j\in\Z,\  \lambda \in \LL, \psi \in \Psi}.
\end{equation*}
Here, $d(\LL)=|\det P|$ is the {\it determinant} of the lattice $\LL=P\Zn$ for some $P\in GL_n(\R)$. Note that $0< d(\LL)\le 1 \le d(\LL^\ast)$ and that $d(\LL) = 1$ only when $\LL = \Zn$.

Given a matrix $B \in GL_n(\R)$ and a lattice $\LL$, we
define a countable subgroup of $\Rn$ by
\[
\sum_{j\in \Z}    B^j \LL^* = \{ x \in \Rn: x = \sum_{j\in \Z} x_j, \ x_j \in B^j \LL^*, \ x_j=0 \text{ for all but finitely many $j$} \}.
\]
Once the dilation $A$ is chosen, our convention is to let $B=A^T$. We shall prove our main oversampling result under the assumption that
  \begin{equation}
    \label{eq:condition-O-strong}
\boxed{\;  \biggl(\sum_{j\in \Z}    B^j \LL^* \biggr)  \cap \Zn \subset \LL^\ast.}  
  \end{equation}
To achieve this we shall establish the following key lemma as a consequence of Lemma \ref{sacr} and Theorem \ref{adu}.

\begin{lemma}\label{approx}
  Let $A \in GL_n(\R)$ be expansive and let $\LL$ be a lattice in $\Rn$ containing $\Zn$ and satisfying condition~\eqref{eq:condition-O-strong}. Then, for any $J\in\N$ and $\ve>0$, there exists a set $K=K_{J,\ve}$, which is an $\ve$-approximate transversal constellation of $A^j\LL/A^j\Zn$ for all $|j| \le J$.
  \end{lemma}

\begin{proof}
Fix $J\in \N$ and let $F = \bigcup_{|j| \le 2J} B^j F_0$, where $F_0$ is a basis of the lattice $\LL^*$. By Proposition \ref{elm} one can show that, for any $\ve>0$, there exists $\delta>0$ such that
\[
F^{*,\delta} \subset \bigcap_{|j| \le 2J} (A^j \LL + B(0,\ve)) .
\]
Thus, applying Theorem \ref{adu} to $G=
\sum_{|j| \le 2J}    B^j \LL^*$ and using $G \cap \Zn \subset \LL^*$ yield 
\[
\LL \subset \bigcap_{|j| \le 2J} (A^j \LL + B(0,\ve)) +\Zn.
\]
For any $\ve>0$ we can find $\ve'>0$ such that $A^i(B(0,\ve')) \subset B(0,\ve)$ for $|i|\le J$. Thus, the above formula can be strengthened to
\[
A^i\LL = \bigcap_{|j| \le 2J} (A^{i+j} \LL + A^i(B(0,\ve'))) +A^i\Zn \subset \bigcap_{|j| \le J} (A^j \LL + B(0,\ve)) +A^i\Zn \qquad\text{for all }|i| \le J.
\]
For $|i| \le J$, we define lattices  $\LL_i = A^i\LL$ and $\LG_i =
A^i\Zn $. By Lemma \ref{sacr}, for each $\ve>0$ there exists a set
$K=K_{J,\ve}$ with cardinality $\card{K}=(\card{\LL/\Zn})^{2J+1}$, which is the $\ve$-approximate transversal constellation of each $\LL_i/\LG_i$, $|i| \le J$.
\end{proof}

\subsection{Second Oversampling Theorem for real dilations}
\label{sec:second-overs-theor}
Our main oversampling result takes the following simple form.

\begin{theorem}\label{thm:oversamp-frames}
  Let $A \in GL_n(\R)$ be expansive, $B=A^T$, and $\Psi \subset L^2(\Rn)$. Take
  $\LL \supset \Zn$ to be a lattice in $\Rn$ satisfying
  \eqref{eq:condition-O-strong}. If $\af{\Psi,A,\Zn}$ is a frame with bounds $C_1$ and $C_2$, then so
is $\af{d(\LL)^{1/2} \Psi,A,\LL}$. 
\end{theorem}

Theorem \ref{thm:oversamp-frames} automatically implies a more general result for wavelet systems associated with an arbitrary dilation lattice $\LG$.

\begin{corollary}\label{gof}
  Let $A \in GL_n(\R)$ be expansive, $B=A^T$, $\Psi \subset L^2(\Rn)$, and $\LG$ be a full rank lattice. Assume that the oversampling lattice $\LL \supset \LG$ satisfies
\begin{equation}\label{gen-O-strong}
 \biggl( \sum_{j\in \Z} B^j \LL^* \biggr) \cap \LG^* \subset \LL^\ast. 
  \end{equation}
If $\af{\Psi,A,\LG}$ is a frame with bounds $C_1$ and $C_2$, then so
is $\af{(d(\LL)/d(\LG))^{1/2} \Psi,A,\LL}$.
\end{corollary}

\begin{proof} Consider the dilation matrix $\tilde A= P^{-1} A P$ and
  the oversampling lattice $\tilde \LL = P^{-1}\LL$, where $\LG =
  P\bzn$ for some $P \in GL_n(\br)$. Then, the transpose $\tilde B=
  \tilde A^T = P^T B (P^T)^{-1}$ and the dual lattice $\tilde \LL^*=
  P^T \LL^*$. 
The condition \eqref{gen-O-strong} implies that \[
 \biggl( \sum_{j\in \Z} \tilde B^j \tilde \LL^* \biggr) \cap \Zn =
 \biggl( \sum_{j\in \Z} P^T B^j \LL^* \biggr) \cap \Zn =
P^T \left( \biggl( \sum_{j\in \Z} B^j \LL^* \biggr) \cap \LG^* \right) \subset
P^T \LL^* = \tilde \LL^\ast. 
\]
Thus, \eqref{eq:condition-O-strong} holds for the dilation $\tilde B$ and the lattice $\tilde \LL$. By our hypothesis, the identity \eqref{eq:stan-gen-trans} implies that $\af{D_P \Psi, \tilde A, \bzn}$ is a frame with bounds $C_1$ and $C_2$. Therefore, Theorem \ref{thm:oversamp-frames} implies that 
\[
\af{d(\tilde \LL)^{1/2} D_P \Psi, \tilde A, \tilde \LL}= D_P ( \af{d(\tilde \LL)^{1/2} \Psi, A, P\tilde \LL}) = D_P(\af{(d(\LL)/d(\LG))^{1/2} \Psi,A,\LL} )
\]
is also a frame with the same bounds. This concludes the proof of Corollary \ref{gof}.
\end{proof}

In order to prove Theorem~\ref{thm:oversamp-frames}, we need the
following variant of \cite[Proposition 3.4] {BLquasi} for an arbitrary translation lattice $\LG$.

\begin{proposition}
 \label{thm:w-af}
 Let $A \in GL_n(\R)$ be expansive, $B=A^T$, $\Psi = \set{\psi_1, \dots,
   \psi_L} \subset L^2(\Rn)$, and $\LG$ be a full rank lattice. Suppose that each $\psi \in \Psi$
 satisfies the local integrability condition    \begin{equation}
 \sum_{j\in\Z} \abs{\hat \psi(B^{-j}\xi)}^2  \in
L^1_{\text{loc}}(\Rn \setminus \{0\}). \label{eq:L1-local-int} \\
\end{equation}
 Then, for each $f \in \D$, the function
  \begin{align*}
    w(x)= \sum_{g \in \af{d(\LG)^{1/2}\Psi,A,\LG}}\abs{\innerprod{\tran[x]f}{g
      }}^2 = d(\LG) \sum_{l=1}^L \sum_{j \in \Z} \sum_{\gamma \in \LG}
    \abs{\innerprod{\tran[x]f}{\dila \tran[\gamma] \psi_l }}^2
    \end{align*}
is an almost periodic function that coincides pointwise with the absolutely
    convergent series
\begin{equation}
  \label{eq:w-FS-af}
    w(x)= \sum_{l=1}^L \sum_{j \in \Z} \sum_{m\in \LG^\ast} 
    c_{j,l}(m) \mathrm{e}^{2 \pi i \innerprods{B^j m}{x}},
\end{equation}
where 
\begin{equation*}
  c_{j,l}(m) =  \int_{\Rn} \hat f(\xi) \overline{\hat
    f(\xi+B^j m)} \, \overline{\hat \psi_l(B^{-j}\xi)} \hat
  \psi_l(B^{-j}(\xi+B^j m))\, \mathrm{d}\xi.
\end{equation*} 
\end{proposition}

We also use the notation:
\begin{equation}
 N(\tran[x]f,\LG) = \sum_{g \in
   \af{d(\LG)^{1/2}\Psi,A,\LG}}\abs{\innerprod{\tran[x]f}{g}}^2, \label{eq:3}
 \end{equation}
hence $N(f,\LG)=w(0)=\sum_{l=1}^L \sum_{j \in \Z} \sum_{m\in \LG^\ast} c_{j,l}(m)$.

The proof of Theorem~\ref{thm:oversamp-frames} relies on the following
key result on translational averaging of wavelet functionals motivated
by the results of Laugesen \cite{La02}. This theorem is a consequence
of our results on the existence of simultaneous approximate
transversal constellations, see Lemma \ref{approx}.

\begin{theorem}\label{thm:averaging}
  Let $A \in GL_n(\R)$ be expansive, $\Psi \subset L^2(\Rn)$, and let $\LL$ be an
  lattice in $\Rn$ containing $\Zn$ and satisfying condition~\eqref{eq:condition-O-strong}.  
Suppose that the wavelet system $\af{\Psi,A,\Zn}$ is a Bessel sequence. Then, there exists a sequence $\{D_J\}_{J\in\N}$ of finite subsets of $\Rn$ such that
\begin{equation}
  \label{eq:funct-averaging}
  N(f,\LL) = \lim_{J \to \infty} \frac{1}{\card{D_J}} \sum_{d \in D_J}
  N(\tran[d]f, \Zn)  \qquad \text{for $f \in \D$,} 
\end{equation}
where $\D$ is given by \eqref{eq:def-D} and $N$ by
\eqref{eq:3}.
\end{theorem}

\begin{proof}
Since $\af{\Psi,A,\Zn}$ is a Bessel sequence, the series in \eqref{eq:L1-local-int} defines a bounded function; see \cite{BLquasi, HLW}. Therefore, we can freely apply Proposition \ref{thm:w-af}.

  Fix $J \in \N$ and $f \in\D$. Let $\eps >0$, and let $K=K_{J,\ve}$
  be an $\ve$-approximate transversal constellation of $A^j\LL/A^j\Zn$
  for all $|j| \le J$ as provided by Lemma~\ref{approx}. 
   We want
  to express $N(f,\LL)$ as an average of $N(\tran[d]f, \Zn)$ over
such $\ve$-approximate transversal constellations of $A^j\LL/A^j\Zn$. 
Thus, we consider
\begin{align} \nonumber
 \frac{1}{\card{K}} \sum_{d \in K}
  N(\tran[d]f, \Zn) &= 
\frac{1}{\card{K}} \sum_{d \in K} 
  \sum_{l=1}^L \sum_{\abs{j} \le J} \sum_{m \in \Zn} c_{j,l}(m)
  \,\mathrm{e}^{2\pi i \innerprod{B^jm}{d}} \\ &+ \frac{1}{\card{K}}  \sum_{d \in K}
  \sum_{l=1}^L \sum_{\abs{j} > J} \sum_{m \in \Zn} c_{j,l}(m)
  \,\mathrm{e}^{2\pi i \innerprod{B^jm}{d}} \nonumber \\ &=: I_1(J)
  + I_2(J), \label{eq:7}
\end{align}
which follows by \eqref{eq:w-FS-af}. 
By the absolute convergence of the sum above, we conclude that $ I_2(J)
\to 0$ as $J\to \infty$ regardless of the choice of $K$. Let $\delta_{m,\LL^*}=1$ if $m\in \LL^*$ and $0$ otherwise. By Lemma~\ref{apple} we have that
\begin{align*}
\bigg| I_1(J) - \sum_{l=1}^L \sum_{\abs{j} \le J} \sum_{m \in \LL^*}
  c_{j,l}(m) \bigg| &= \bigg| \sum_{l=1}^L \sum_{\abs{j} \le J} \sum_{m \in \Zn}
  c_{j,l}(m) \bigg( \frac{1}{\card{K}} \sum_{d \in K} \mathrm{e}^{2\pi i
    \innerprod{B^jm}{d}} - \delta_{m,\LL^*} \bigg) \bigg| \\
& \le 
 \sum_{l=1}^L \sum_{\abs{j} \le J} \sum_{m \in \Zn}
  |c_{j,l}(m)| \min\{ O(\abssmall{B^jm} \eps), 1\} \to 0 \qquad
\end{align*}
as $\eps \to 0$. Indeed, the last step follows from the absolute convergence of the series $\sum_{l,j,m} c_{j,l}(m)$ and the Lebesgue Dominated Convergence Theorem. Consequently, we can find a sequence
$\{D_J\}_{J\in\N}$ of finite subsets of $\Rn$ defined by $D_J=K_{J,\eps}$ for some sufficiently small $\eps=\eps(J)>0$ such that
\[ \frac{1}{\card{D_J}} \sum_{d \in D_J}
  N(\tran[d]f, \Zn) =I_1(J)+I_2(J) \to \sum_{l=1}^L \sum_{j \in \Z} \sum_{m \in \LL^*}
  c_{j,l}(m) =  N(f,\LL) 
\]
as $J \to \infty$. This completes the proof of Theorem \ref{thm:averaging}.
\end{proof}

\begin{proof}[Proof of Theorem~\ref{thm:oversamp-frames}]
  Assume that the wavelet system $\af{\Psi,A,\Zn}$ is a frame for $L^2(\Rn)$
  with bounds $C_1, C_2$. 
 By our hypothesis there are constants $C_1, C_2 > 0$ so that
\begin{align*}
  C_1 \norm{f}^2 \le N(f,\Zn) \le C_2 \norm{f}^2 \qquad \forall f \in L^2(\Rn).
\end{align*}
Let $\{D_J\}_{J \in \N}$ be a sequence of finite subsets such that
\eqref{eq:funct-averaging} holds. 
Fix $J \in \N$. For each
$d \in D_J$ we have
\begin{align*}
  C_1 \norm{f}^2 \le N(\tran[ d]f, \Zn) \le C_2 \norm{f}^2 \qquad \forall f \in L^2(\Rn),
\end{align*}
where we have used that $\norm{\tran[x] f} =\norm{f} $ for $x \in
\Rn$. Adding these inequalities for each $d \in D_J$ yields:
\begin{align*}
  \card{D_J}\, C_1 \norm{f}^2 \le \sum_{d \in D_J}
  N(\tran[d]f, \Zn) \le \card{D_J}\, C_2
  \norm{f}^2. 
\end{align*} 
Taking the limit $J \to \infty$ 
gives us 
\begin{align*}
  C_1 \norm{f}^2 \le \lim_{J\to \infty} \frac{1}{\card{D_J}} \sum_{d \in D_J}
  N(\tran[d]f, \Zn) \le C_2 \norm{f}^2
\end{align*}
for all $f \in L^2(\Rn)$. By an application of Theorem~\ref{thm:averaging}, we arrive at 
\begin{align*}
  C_1 \norm{f}^2 \le  N(f,\LL)   \le C_2 \norm{f}^2 \quad \text{for } f \in \D.
\end{align*}
Extending these inequalities to all of
$L^2(\Rn)$ by a standard density argument completes the proof of Theorem \ref{thm:oversamp-frames}.
 \end{proof}

\begin{example}[Rational dilations in one dimension]
 \label{ex:rat-dila-one-dim}
  Let $A=B^T=p/q \in \Q$ for
  relatively prime $p,q \in \N$, and let $\LL = 1/\lambda\, \Z$ for some
  $\lambda \in \Z$. Then $\LL^\ast = \lambda \Z$, and
  condition~\eqref{eq:condition-O-strong} reads
  \begin{equation}\label{eq:example1-strong}
 \biggl( \sum_{j\in \Z} \Bigl( \frac{p}{q}\Bigr)^j \lambda \Z \biggr) \cap \Z
\subset \lambda \Z. 
\end{equation}
Since, for $J \in \Z$, 
\[  \sum_{\abs{j} \le J } \Bigl( \frac{p}{q}\Bigr)^j \Z =
\frac{1}{(pq)^J}\, \Z, \]
condition~\eqref{eq:example1-strong} is equivalent to 
   \begin{alignat*}{2}
  \lambda \Z \cap (pq)^J \Z
 &\subset \lambda (pq)^J \Z  \qquad  &\forall J \in\Z ,
 \intertext{or in other words,}
   \lcm(\lambda, (pq)^J)\, \Z
  &\subset \lambda (pq)^J \Z   &\forall J \in\Z.
 \end{alignat*}
  Hence, $\lambda$ needs to be relatively prime to $(pq)^J$ for $J \in
  \N$, or simply, relatively prime to $p$ and $q$. In this case,  $B=p/q$ and $\LL^\ast=\lambda \Z$
satisfy \eqref{eq:condition-O-strong}. 
  Thus, if $\af{\Psi,p/q,\Z}$ is a frame with bounds $C_1$ and $C_2$,
  then so is $\af{\lambda^{-1/2} \Psi,p/q,1/\lambda \Z}$ whenever
  $\lambda$ is relatively prime to $p$ and $q$. However, note that such $B=p/q$ and $\LL^\ast =\lambda \Z$ do not satisfy either of the conditions (i)--(iv) in
  Proposition~\ref{thm:assump-on-LL} in the next section, in particular, $B \LL^\ast \not
  \subset \LL^\ast$. Therefore, none of the previously known equivalent conditions on oversampling lattices in the integer setting, which are described in the next subsection, are satisfactory for non-integer dilations. 
\end{example}

\subsection{Related work for integer and rational dilations}
\label{sec:relat-work}

The Second Oversampling Theorem is well-known for \emph{integer},
expansive dilations and assumptions on the oversampling lattice $\LL$ as 
 \eqref{eq:cond-L} or \eqref{eq:cond-RS} below. 
We briefly review the relationship between previous results and Theorem \ref{thm:oversamp-frames}.

Laugesen \cite[Theorem 6.1]{La02} proved the Second Oversampling Theorem
under the assumption that $A$ is prime relative to $\LL$ and
that $A$ preserves the lattice $\LL$: 
\begin{equation}
  \label{eq:cond-L}
B \Zn \cap \LL^\ast \subset B \LL^\ast \subset \LL^\ast .
\end{equation}
The result in \cite[Theorem 6.1]{La02} is stated for both
expansive and amplifying dilations; in this paper we only consider
expansive matrices.

The formulation of Ron and Shen \cite[Theorem 4.19]{RS97}  of the same result uses
the assumption 
\begin{equation}
\label{eq:cond-RS}
  B^j \Zn \cap \LL^\ast = B^j \LL^\ast \qquad \text{for all $j \ge 0$.}  
\end{equation}
This condition is equivalent to \eqref{eq:cond-L}. It is obvious (take
$j=1$) that \eqref{eq:cond-RS} implies that $B \LL^\ast \subset \LL^\ast$.

Yet another equivalent set of assumptions is used  in the formulation
of the Second Oversampling Theorem of Johnson \cite[Theorem
3.2]{MR2048403}: 
\begin{equation}
A \LL \subset \LL \qquad\text{and}\qquad
A^{-1}\Zn \cap \LL = \Zn. \label{eq:brody-2}
\end{equation}
Note that the assumption $\Zn \subset \LL$
in \cite[Definition 2.2]{MR2048403} is not necessary since it is
implied by \eqref{eq:brody-2}. 

Finally, Chui and Sun \cite{CSun1, CSun2} have completed the theory of oversampling of tight affine systems with integer dilations. In \cite[Theorem 4.1]{CSun2} they characterized lattices $\LL$ for which the conclusion of the Second Oversampling Theorem holds. The oversampling by $\LL$ preserves tightness  for all tight affine frames if and only if $\LL$ satisfies 
\eqref{eq:cond-RS} and hence any of its equivalent forms listed above. Moreover,  Chui and Sun characterized the preservation of tightness for fixed generators $\Psi$ of tight affine frames in terms of explicit equations involving generators in the frequency domain and identities involving so-called ``oversampled frame operators'' in the space domain.

In light of the results of Chui and Sun \cite{CSun2}, it is not surprising that, for integer dilations, all of the previously studied conditions on tightness preserving lattices $\LL$ are equivalent to our newly introduced condition \eqref{eq:condition-O-strong}. We state these conditions in the proposition below.

\begin{proposition}\label{thm:assump-on-LL}
Suppose that $A=B^T \in M_n(\Z)$ is invertible and $\Zn \subset \LL$. Then, the following assertions are equivalent:
  \begin{enumerate}[(i)]
\item     $B \Zn \cap \LL^\ast \subset B \LL^\ast \subset \LL^\ast$,
\item $ B^j \Zn \cap \LL^\ast = B^j \LL^\ast$ for all $j \ge 0$,
\item $A \LL \subset \LL$ and $A^{-1}\Zn \cap \LL = \Zn$,
\item $B \LL^\ast \subset \LL^\ast$ and $(\LL^\ast
  \setminus B\LL^\ast) \subset (\Zn \setminus B\Zn)$,
\item   $ B^j \Zn \cap  \LL^\ast
\subset B^j \LL^\ast$ for all $j \in \Z$.
\item  $\big(\sum_{j\in \Z}    B^j \LL^* \big)  \cap \Zn \subset \LL^\ast$.  
  \end{enumerate}
\end{proposition}

\begin{proof}
  (i) $\implies$ (ii): For $j=0$ there is nothing to prove since
  $ \LL^\ast \subset \Zn$. For $j =1$ we only need to prove $B\Zn \cap
  \LL^\ast \supset B\LL^\ast$, but this follows from $\LL^\ast \subset
  \Zn$ and $B\LL^\ast \subset \LL^\ast$. We now prove (ii) for $j=2$.
  Using $B\Zn \subset \Zn$ and $B \Zn \cap \LL^\ast \subset B \LL^\ast $
  we find 
  \begin{align*}
    B^2 \Zn \cap  \LL^\ast &=  B^2 \Zn \cap  \LL^\ast \cap B\Zn \\
    &\subset   B^2 \Zn \cap  B\LL^\ast =    B (B\Zn \cap  \LL^\ast) \\  
    &\subset   B (B \LL^\ast) = B^2 \LL^\ast  .
  \end{align*}
 The other inclusion follows from:
  \begin{align*}
B^2 \LL^\ast  &= B (B \LL^\ast)  \subset B  (B\Zn \cap  \LL^\ast) \\
&= B^2\Zn \cap  B\LL^\ast \subset B^2\Zn \cap  \LL^\ast,
  \end{align*}
where we have used  $B \LL^\ast \subset \LL^\ast$ and $B\Zn \cap
  \LL^\ast \supset B\LL^\ast$ from the case $j=1$. The argumentation is
  similar for $j \ge 3$. 

  (ii) $\implies$ (i): Take $j=1$ in (ii), that is, we have $B\Zn
  \cap \LL^\ast = B\LL^\ast$. 
It follows that $B \LL^\ast \subset
  \LL^\ast$ and $B\Zn \cap \LL^\ast \subset B\LL^\ast$. 
  
(ii) $\iff$ (iii) is proved in \cite[p.~636]{MR2048403}; alternatively is (i) $\iff$ (iii) proved in \cite[p.~637]{MR2048403}.  

(i) $\implies$ (iv) is proved in \cite[p.~227]{La02}. (iv)
$\implies$ (i): Suppose $x \in \LL^\ast$. If $x \notin B\LL^\ast$,
then $x \notin B\Zn$. In other words, if $x \in B\Zn$,
then $x \in B\LL^\ast$. Hence, $B \Zn \cap \LL^\ast \subset B \LL^\ast$.

(i) $\implies$ (v): This implication is immediate for $j\ge 0$ by (ii). If $j<0$, then (v) is equivalent to $\Zn \cap B^{-j}\LL^* \subset \LL^*$, which holds by $B\LL^* \subset \LL^*$.

(v) $\implies$ (i): Taking $j=1$ and $j=-1$ in (v) shows the first and the second inclusion in (i), respectively.

(v) $\iff$ (vi): Take $x \in \sum_{j\in \Z} B^j \LL^*$. Since (v) $\implies$ (i), we know that $B \LL^* \subset \LL^*$. Hence, $x \in B^{j_0} \LL^*$ for some $j_0 \in \Z$. This shows the implication (v) $\implies$ (vi). The opposite implication is trivial. 
\end{proof}

\begin{example}[Rational dilations in higher dimensions]
Hern\' andez, Labate, Weiss, and Wilson \cite[Theorem 2.12]{HLWW} have
proved the Second Oversampling Theorem for a class of \emph{rational},
expansive dilations $A= PQ^{-1} \in GL_n(\Q)$, where $P,Q \in M_n(\Z)$
are invertible, and $P$ commutes with $Q$. The assumptions on the
oversampling lattice are
\begin{equation}
 P^T \Zn \cap \LL^\ast = P^T \LL^\ast, \qquad Q^T \Zn \cap
\LL^\ast = Q^T \LL^\ast,\label{eq:4}
\end{equation}
which are higher dimensional analogues of
Example~\ref{ex:rat-dila-one-dim}. We remark that in \cite[Theorem
2.12]{HLWW} it is also assumed that $R P R^{-1}, RQ R^{-1} \in
M_n(\Z)$ for $\LL = R^{-1}\Z^n$, but this is equivalent to $P^T\LL^\ast \subset \LL^\ast$ and $ Q^T \LL^\ast \subset \LL^\ast$, and
therefore follows from (\ref{eq:4}). 

Theorem 2.12 in \cite{HLWW} is in fact a special case of
Theorem~\ref{thm:oversamp-frames}. To see this assume that the
assumption of Theorem 2.12 in \cite{HLWW} holds and note that by Proposition~\ref{thm:assump-on-LL}  
(i)~$\Rightarrow$~(iv) it then follows that
\begin{align}
  (\LL^\ast
  \setminus P^T\LL^\ast) &\subset (\Zn \setminus P^T\Zn)\label{eq:prop36iv-Pt}\\
\intertext{and}  (\LL^\ast \setminus Q^T\LL^\ast) &\subset (\Zn \setminus Q^T\Zn).\label{eq:prop36iv-Qt}
\end{align}
By commutativity of $P^T$ and $Q^T$, equation~(\ref{eq:prop36iv-Pt})
implies 
\[ (Q^T \LL^\ast
  \setminus P^T Q^T\LL^\ast) \subset (Q^T\Zn \setminus P^TQ^T\Zn) \subset (\Zn \setminus P^TQ^T\Zn),\] 
which in turn implies that
\begin{align*}
   (\LL^\ast \setminus P^T Q^T\LL^\ast) &\subset 
   \bigl[(\LL^\ast\setminus Q^T \LL^\ast ) \cup Q^T \LL^\ast \bigr]\setminus P^T
   Q^T\LL^\ast \\ &\subset \bigl[(\Zn \setminus Q^T \Zn ) \cup Q^T \LL^\ast \bigr]\setminus P^T
   Q^T\LL^\ast \\ &\subset (\Zn \setminus Q^T \Zn ) \cup (\Zn \setminus P^TQ^T\Zn ) \\
&=  (\Zn \setminus P^TQ^T\Zn)
\end{align*}
where we have used (\ref{eq:prop36iv-Qt}) in the second step. We have
showed that
\[   (\LL^\ast
  \setminus P^TQ^T\LL^\ast) \subset (\Zn \setminus P^TQ^T\Zn)
\]
holds, which (by Proposition~\ref{thm:assump-on-LL}
(iv)~$\Rightarrow$~(v)), for $B=P^TQ^T$, implies that
\[   (P^TQ^T)^j \Zn \cap \LL^\ast
  \subset (P^TQ^T)^j\LL^\ast  \qquad \text{for all $j \in \Z$},
\]
that is, 
\[   \Zn \cap (P^TQ^T)^{-j} \LL^\ast  \subset \LL^\ast  \qquad \text{for all $j \in \Z$}.
  \]
Since $(P^T)^{-J}\LL^\ast = \sum_{\abs{j} \le J} (P^T)^j \LL^\ast$ for any $J \in
\N$, we have, in particular, that 
\[   \Zn \cap (Q^T)^{-J}\sum_{\abs{j} \le J} (P^T)^j \LL^\ast 
  \subset \LL^\ast ,
\]
and by the commutativity of $Q^T$ and $P^T$ that
\[   \Zn \cap \sum_{\abs{j} \le J} (P^T)^j(Q^T)^{-J} \LL^\ast 
  \subset \LL^\ast.
\]
Since $ (Q^T)^{j} \LL^\ast \subset (Q^T)^{-J} \LL^\ast
$ for $\abs{j} \le J$, it then follows that 
\[ \Zn \cap
\sum_{\abs{j} \le J} (P^T)^j(Q^T)^{-j} \LL^\ast \subset \LL^\ast ,
\]
which, since $J>0$ is arbitrary, implies
 \[ \Zn \cap
\sum_{j \in \Z} (P^T)^j(Q^T)^{-j} \LL^\ast \subset \LL^\ast.
\]
Using commutativity of $P^T$ and $Q^T$, the last equation implies that
(\ref{eq:condition-O-strong}) holds.
\end{example}

\subsection{Oversampling with the support condition}
\label{sec:overs-with-supp}
In the following theorem we relax the condition~\eqref{eq:condition-O-strong}
from Theorem~\ref{thm:oversamp-frames} by supposing a support
condition.  The result resembles somewhat Theorem 3 in \cite{CS} referred to as the First Oversampling Theorem.

\begin{theorem}\label{thm:oversamp-frames-support}
  Let $A \in GL_n(\R)$ be expansive and $\Psi \subset L^2(\Rn)$. For
  $J_0 \in \N_0$ take
  $\LL \supset \Zn$ to be a lattice in $\Rn$ satisfying:
  \begin{equation}
    \label{eq:condition-J-strong}  
 \biggl(\sum_{j\in \Z}    B^j \LL^* \biggr)  \cap \Zn \subset B^{-J_0}\LL^\ast. 
  \end{equation}
Suppose that every $\psi \in \Psi$ satisfies the support condition:
\begin{equation}
  \label{eq:condition-supp}
  \hat \psi(\xi) \hat \psi (\xi + k)=0 \qquad \text{for all $k \in \Zn
    \setminus B^{J_0} \Zn$.}
\end{equation}
Then, if $\af{\Psi}$ is a frame with bounds $C_1$ and $C_2$, so is $\af{d(A^{-J_0}\LL)^{1/2} \Psi,A,A^{-J_0}\LL}$. 
\end{theorem}

 \begin{proof}
We shall only sketch the proof since it is similar to that  of
Theorem~\ref{thm:oversamp-frames}. Again, the key ingredient is to show the translational averaging formula
  \eqref{eq:funct-averaging}. 
  Since $\af{\Psi,A,\Zn}$ is a Bessel sequence,  we can freely apply Proposition \ref{thm:w-af}. 
  
First observe that \eqref{eq:condition-J-strong} implies that $\LL^* = \LL^* \cap \Zn \subset B^{-J_0}\LL^*$. By taking duals, we have $\LL \supset A^{J_0} \LL$ and thus $\LL \subset A^{-J_0}\LL$. Hence, it is meaningful to talk about the quotient groups $A^{j-J_0}\LL/A^j\Zn$ for $j\in\Z$.
Fix $J \in \N$. Mimicking the proof of Lemma~\ref{approx} we can show that for all $\eps >0$, there exists $K=K_{J,\ve}$ which an $\ve$-approximate transversal constellation of $A^{j-J_0}\LL/A^j\Zn$
  for all $|j| \le J$. 

Fix $f \in\D$. As in the proof of Theorem~\ref{thm:averaging} we consider $I_1(J)$ and $I_2(J)$ defined in equation~\eqref{eq:7}. As before we have $I_2(J) \to 0$ as $J\to\infty$. Since $c_{j,l}(m)=0$ for
  all $m \in \Zn \setminus B^{J_0} \Zn$ by \eqref{eq:condition-supp}, we
  have that
\[
  I_1(J) = \sum_{l=1}^L \sum_{\abs{j} \le J} \sum_{m \in (B^{J_0} \Zn) \cap \Zn }
  c_{j,l}(m) \frac{1}{\card{K}} \sum_{d \in K} \mathrm{e}^{2\pi i \innerprod{B^jm}{d}} 
\]
As in the proof of Theorem~\ref{thm:averaging} by Lemma \ref{apple} one can show that
\[
I_1(J)  \to \sum_{l=1}^L \sum_{\abs{j} \le J} \sum_{m \in (B^{J_0} \LL^*) \cap \Zn}
  c_{j,l}(m) 
 = \sum_{l=1}^L \sum_{\abs{j} \le J} \sum_{m \in B^{J_0} \LL^*}
  c_{j,l}(m)
\]
as $\eps\to 0$. The last step is a consequence of $B^{J_0} \LL^* \subset \LL^* \subset \Zn$.
It follows that we can find a sequence
$\{D_J\}_{J\in\N}$ of finite subsets of $\Rn$ such that 
\[ \frac{1}{\card{D_J}} \sum_{d \in D_J}
  N(\tran[d]f, \Zn) \to \sum_{l=1}^L \sum_{j \in \Z} \sum_{m \in B^{J_0} \LL^*}
  c_{j,l}(m) =  N(f,A^{-J_0}\LL) 
\]
as $J \to \infty$.
 The rest of proof goes along the lines of the proof of  Theorem \ref{thm:oversamp-frames}.
 \end{proof}

\begin{remark}\label{r2} Observe that \eqref{eq:condition-J-strong} implies the following weaker condition  \begin{equation}
B^{J_0+j} \Zn \cap \LL^\ast \subset B^j \LL^\ast  \qquad \text{for all $j\in \Z$}.\label{eq:condition-J-weak}
  \end{equation}
Moreover, under the extra assumptions $A \in M_n(\Z)$ and $B
  \LL^\ast \subset \LL^\ast$ one can show, by replicating the proof of
  Proposition \ref{thm:assump-on-LL}, that
  condition~\eqref{eq:condition-J-strong} is equivalent with
  \begin{equation}
    B^{J_0+1} \Zn \cap \LL^\ast  \subset B
    \LL^\ast.   \label{eq:condition-J-weak-2} 
 \end{equation}
 Recall that under these assumptions condition~(\ref{eq:condition-O-strong}) from
 Theorem~\ref{thm:oversamp-frames} is equivalent to $ B \Zn \cap
 \LL^\ast \subset B \LL^\ast$ which is  more restrictive on $A$
 and $\LL$ than \eqref{eq:condition-J-weak-2}. Indeed, for $n=1$ with $A=a>1$ and
 $\LL =\Z/\lambda$, where $a,\lambda \in \N$, condition
 (\ref{eq:condition-O-strong}) is satisfied if and only if $a$ and
 $\lambda$ are relative prime, while \eqref{eq:condition-J-weak-2} and hence
 \eqref{eq:condition-J-strong} are satisfied exactly when $a\lambda$
 divides $\lcm(a^{J_0+1},\lambda)$. In particular, for any given
 $a,\lambda \in \N$, we can always find a $J_0 \in \N_0$ such that
 \eqref{eq:condition-J-strong} is satisfied.
\end{remark}

\section{Oversampling of dual frames for real dilations }
\label{sec:dual}

In this section we establish the analogues of Theorems \ref{thm:oversamp-frames} and \ref{thm:oversamp-frames-support} for dual affine frames. We also give a counterexample to a result of Chui, Czaja, Maggioni, and Weiss \cite{CCMW} on the oversampling of rationally dilated dual affine frames.

For Bessel affine systems $\af{\Psi,A,\LG}$ and $\af{\Phi,A,\LG}$, we define, for each $\alpha
    \in \Zn$:
\begin{equation*}
 {t}^\LG_\alpha(\xi) = \sum_{l=1}^L \sum_{\sumstyle{j\in\Z}{B^{-j}\alpha
     \in \LG^\ast}} \hat \psi_l(B^{-j}\xi)
    \overline{\hat \phi_l(B^{-j}(\xi+\alpha))}. 
    \end{equation*}
It is well known that two Bessel families $\af{\Psi,A,\LG}$ and $\af{\Phi,A,\LG}$ are dual
frames \ifft{} $t^\LG_\alpha(\xi) = \delta_{\alpha,0}$ for almost every $\xi$
and all $\alpha \in \Zn$. The proof of this result can be found in \cite[Theorem 4]{CCMW} and \cite[Theorem 9.6]{HLW}. 

\subsection{The Second Oversampling Theorem for dual frames}
\label{sec:2nd-dual}

Before we present the main results of this section, we introduce
yet another condition on the oversampling lattice $\LL \supset \Zn$:
  \begin{equation}
    \label{eq:condition-O-weak}
 \boxed{ \,   B^j \Zn \cap \LL^\ast \subset B^j \LL^\ast  \qquad \text{for all $j
     \in \Z$}.}   
  \end{equation}
This new assumption on $\LL$ is obviously weaker than \eqref{eq:condition-O-strong}.
The following result can then be seen as an analogue of Theorem \ref{thm:oversamp-frames} for dual affine frames.

\begin{theorem}\label{thm:oversamp-dual-frames}
  Let $A \in GL_n(\R)$ be expansive and $\Psi, \Phi \subset L^2(\Rn)$. 
Suppose that either of the following assertions holds.
\begin{enumerate}[(i)]
\item The oversampling lattice $\LL \supset \Zn$ satisfies
  \eqref{eq:condition-O-weak} \emph{and} the oversampled affine systems
  $\af{d(\LL)^{1/2} \Psi,A,\LL}$ and $\af{d(\LL)^{1/2} \Phi,A,\LL}$
  are Bessel sequences.
\item  The oversampling lattice $\LL \supset \Zn$ satisfies
  \eqref{eq:condition-O-strong}
\end{enumerate}
Then, if $\af{\Psi}$ and $\af{\Phi}$ are dual frames, so are
$\af{d(\LL)^{1/2} \Psi,A,\LL}$ and $\af{d(\LL)^{1/2} \Phi,A,\LL}$. 
\end{theorem}

\begin{proof}
(i):  By \cite[Theorem 4]{CCMW} it suffices to prove that
  $t^\LL_\alpha(\xi) = \delta_{\alpha,0}$ for $\alpha \in \Zn$. From
  our hypothesis we know that $t^{\Zn}_\alpha(\xi) = \delta_{\alpha,0}$.
  Fix $\alpha \in \Zn$. We first note that 
  \begin{equation}
 \setprop{j\in\Z}{B^{-j}\alpha \in \LL^\ast} \subset \setprop{j\in\Z}{B^{-j}\alpha \in \Zn},\label{eq:6}
 \end{equation}
since $\LL^\ast \subset \Zn$. Then, we claim that equality between the above 
sets holds when  $\setprop{j\in\Z}{B^{-j}\alpha \in \LL^\ast}$ is
non-empty. To see this take $j \in \Z$ so that $B^{-j}\alpha \in \Zn$.
By assumption there is a $j_0 \in \Z$ such that $B^{-j_0}\alpha \in
\LL^\ast$. Thus, by \eqref{eq:condition-O-weak},
\begin{align*} 
  \alpha \in B^j \Zn \cap B^{j_0}\LL^\ast = B^{j_0}(B^{j-j_0} \Zn \cap
  \LL^\ast) \subset B^{j_0}(B^{j-j_0}  \LL^\ast) = B^j \LL^\ast,
 \end{align*}
that is, $B^{-j}\alpha \in \LL^\ast$.

If $\setprop{j\in\Z}{B^{-j}\alpha \in \LL^\ast}=\emptyset$
for some (non-zero) $\alpha \in \Zn$, then trivially $t^\LL_\alpha(\xi)
= 0$. On the other hand, if $\setprop{j\in\Z}{B^{-j}\alpha \in
  \LL^\ast}$ is non-empty, then $t^\LL_\alpha(\xi) =
t^{\Zn}_\alpha(\xi) = \delta_{\alpha,0}$ by the claim above. The conclusion
is that $t^\LL_\alpha(\xi) = \delta_{\alpha,0}$ for $\alpha \in \Zn$.

(ii): By Theorem~\ref{thm:oversamp-frames} the oversampled
systems $\af{d(\LL)^{1/2} \Psi,A,\LL}$ and $\af{d(\LL)^{1/2}
  \Phi,A,\LL}$ are frames and hence, in particular, Bessel sequences. The
result now follows directly from Theorem~\ref{thm:oversamp-dual-frames}(i).
\end{proof}

As a direct consequence of the proof of Theorem~\ref{thm:oversamp-dual-frames}(i) we
have the following oversampling result for Parseval (tight) frames.
\begin{theorem}\label{thm:oversamp-tight-frames}
  Let $A \in GL_n(\R)$ be expansive and $\Psi \subset L^2(\Rn)$. 
Take
  $\LL \supset \Zn$ to be a lattice in $\Rn$ satisfying \eqref{eq:condition-O-weak}.
 If $\af{\Psi}$ is a Parseval frame, then so is
 $\af{d(\LL)^{1/2} \Psi,A,\LL}$.
 \end{theorem}

\subsection{Related work}
Laugesen \cite[Theorem 8.3]{La02} has proved the Second
Oversampling Theorem for dual frames for \emph{integer} (and expansive
or amplifying) dilations $A \in M_n(\Z)$ under the assumption
\eqref{eq:cond-L} and, as usual, $\Zn \subset \LL$. Within the settings of expansive dilations, Theorem 8.3 in
\cite{La02} is therefore a special case of Theorem~\ref{thm:oversamp-dual-frames}.

Chui, Czaja, Maggioni, and Weiss have three versions
of the  Second Oversampling Theorem for dual frames \cite[Proposition
1]{CCMW}. In our notation, their results can be summarized as follows:
\begin{enumerate}[(i)]
\item The first result uses the assumptions \cite[(4.1) \& (4.2)]{CCMW} and that some
  power of $B$ preserves the lattices $\Zn$ and $\LL^\ast$ (a priori need not be the same power), \ie there exist $\ga,\ga' \in\bn$ such that
\begin{align}
 &B^\ga\bzn \subset \bzn, \qquad \bzn \cap B^j \bzn = \{0\} \text{ for } 0<j<\ga,\label{c1}
 \\
 &B^{\ga'} \LL^*\subset \LL^*, \qquad \LL^* \cap B^j \LL^* = \{0\} 
 \text{ for } 0<j<\ga',\label{c2}
\\
& \LL^* \setminus B^{\ga'} \LL^* \subset \bzn \setminus B^{\ga} \bzn.\label{c3}
\end{align}
As we will see in Remark \ref{corr} below, the powers $\ga$ and $\ga'$
must be actually equal, and consequently our condition
\eqref{eq:condition-O-weak} will hold. In particular, the hypothesis \cite[(4.1)]{CCMW} is unnecessary. 
 \item The second statement is incorrect (see Example~\ref{ex:ccmw-counterex} below for a counterexample).  

\item The third result uses the assumptions $\LL^\ast \subset \Zn$ and $B^j\Zn \cap \Zn
  = \{0\}$ for all $j > 0 $. This implies $B^j\Zn \cap \Zn
  = \{0\}$ for all $j \neq 0$ which implies our condition~\eqref{eq:condition-O-weak} since $\LL^\ast \subset \Zn$. We also note that the
  condition $A \in E_3(CB)$ in \cite{CCMW}, \ie  $B^j\LL^\ast \cap \LL^\ast
  = \{0\}$ for all $j > 0 $, is implied by $B^j\Zn \cap \Zn
  = \{0\}$ for all $j > 0 $. Hence, $A \in E_3(B) \cap E_3(CB)$ could be replaced by $A
  \in E_3(B)$ in \cite[Proposition 1(iii)]{CCMW}. 
\end{enumerate}
Consequently, the oversampling result in \cite{CCMW} is also a special case of Theorem~\ref{thm:oversamp-dual-frames}.
We also remark that the necessary condition that the oversampled
affine systems are Bessel sequences is missing in all three statements
in \cite[Proposition 1]{CCMW}. This condition can of course be left
out if the second oversampling theorem for frames is available (\eg as in
Theorem~\ref{thm:oversamp-dual-frames}(ii)) or if one is working with
tight frames (\eg as in Theorem~\ref{thm:oversamp-tight-frames}). However, in general we believe the following is an open problem.

\begin{question}
Let $\Psi \subset L^2(\Rn)$ and $A \in GL_n(\R)$ be expansive. Suppose that an affine system
$\af{\Psi}$ is a Bessel sequence and a lattice $\LL \supset \Z^n$. Is the oversampled affine system $\af{d(\LL)^{1/2}\Psi,A,\LL}$ necessarily a
Bessel sequence?
\end{question} 

\begin{remark}\label{corr} Assume that a dilation $B$ and a lattice
  $\LL$ satisfy \eqref{c1}--\eqref{c3}.
By the Smith normal form theorem (see \cite[Theorem
3.7]{TV}) there is a basis  $v_1, \ldots, v_n$ of the lattice
$\LL^*$ and integers $\alpha_i \in \N$ ($i=1,\dots,n$) satisfying $1
\le \alpha_1 \le \cdots \le \alpha_n$ such that  $\alpha_1 v_1,
\ldots, \alpha_n v_n$ is a basis of the lattice
$B^{\ga'}\LL^*$. 
Since $B^{\ga'}\LL^*$ is a proper sublattice of $\LL^*$ ($B$ is
expansive), not all $\alpha_i$ will be equal to one; in particular $\alpha_n
\ge 2$. Define $w_1 = v_1 + v_n , \dots, w_{n-1} = v_{n-1} + v_n$, and
$w_n = v_n$, \ie
considered as (coordinate) column vectors we define $W = VP$, where
\[ W = [w_1 \cdots w_n], \, V = [v_1 \cdots v_n], \, P = 
 \begin{bmatrix}
1 &  0 & 0 & \cdots & 0 \\
0 & 1 & 0 & \cdots & 0 \\
\vdots & &  \ddots & & \vdots\\
0 & 0 & 0 & \ddots & 0 \\
1 & 1 & 1 & \cdots & 1 
\end{bmatrix} .
 \] 
Note that $\LL^\ast = V\Zn = VP\Zn = W\Zn$ since the matrix $P$ is integer valued and
$\det P =1$. Therefore, $\{w_1, \ldots, w_n\}$ is also a basis of   
$\LL^*$, but with the property that $w_i \notin B^{\ga'} \LL^*$ for all $i=1,\ldots,
n$.

Now, by \eqref{c3}, we have $w_i \in \bzn$, and thus $\LL^* \subset
\bzn$. Combining this with \eqref{c1} and \eqref{c2} implies that
$\ga' \ge \ga$. Using the fact that there exists $m\in\bn$ such that
$m\bzn \subset \LL^*$, which \eg follows from formula~(2.3) in \cite{BLquasi} with $m=d(\LL^\ast)$, we can also deduce that $\ga \ge \ga'$.
Thus, conditions \eqref{c1}--\eqref{c3} imply that $B^\ga \in
M_n(\bz)$, $\LL^* \subset \bzn$, $B^\ga \LL^* \subset \LL^*$, and 
\[
\LL^* \setminus B^{\ga} \LL^* \subset \bzn \setminus B^{\ga} \bzn.
\]
By the equivalence (iv) $\iff$ (v) in Proposition
\ref{thm:assump-on-LL} applied for the dilation $B^\ga$ we deduce that
\eqref{eq:condition-O-weak} holds for $j\in \ga\bz$. If $j\not\in \ga\bz$,
then $B^j\bzn \cap \LL^* \subset B^j\bzn \cap \bzn =\{0\}$, and thus
\eqref{eq:condition-O-weak} holds for all $j\in\bz$. 
\end{remark}

\begin{example}
\label{ex:ccmw-counterex}
We consider oversampling of dual frames in  $L^2(\R)$ with dilation
parameter $A = 3/2$. In this setting, Proposition~1(ii) from
\cite{CCMW} states that if $\af{\psi,3/2,\Z}$ and $\af{\phi,3/2,\Z}$
are dual frames, so are $\af{2^{-1/2}\psi,3/2,\Z/2}$ and
$\af{2^{-1/2}\phi,3/2,\Z/2}$; hence, in particular, if
$\af{\psi,3/2,\Z}$ is a Parseval frame, then so is
$\af{2^{-1/2}\psi,3/2,\Z/2}$. We shall exhibit a generator $\psi \in
L^2(\R)$ contradicting this statement. Note that the conclusion from
our Theorem~\ref{thm:oversamp-dual-frames} is that oversampling lattices $\LL = 1/\lambda \Z$ with $\lambda \in \set{1,5,7,11,\dots}$ will guarantee preservation of tightness/duality, see Example~\ref{ex:rat-dila-one-dim}.  
 The definition of $\psi$ is:
  \begin{equation*}
    \hat\psi(\xi) =
    \begin{cases}
      1 & \xi \in \itvco{4/3}{3/2}, \\
      \tfrac{1}{\sqrt{2}}  & \xi \in \itvco{-1}{-2/3} \union \itvco{1}{4/3} \union \itvcc{3/2}{2}, \\
      -\tfrac{1}{\sqrt{2}}  & \xi \in \itvco{-3/2}{-1}, \\
      0 & \text{otherwise,}
    \end{cases}
  \end{equation*}
see also Figure~\ref{fig:ccmw_example}.
\begin{figure}[htbp]
\psset{xunit=3.4cm,yunit=3.cm}
\begin{pspicture}(-2.,-1.)(2.5,1.2)
\psaxes[Dx=1, Dy=1,tickstyle=bottom,linewidth=1.2pt]{->}(0,0)(-2,-1)(2.5,1.2)  
\pspolygon(-1.5,0)(-1.5,-0.707)(-1,-0.707)(-1.,0)
\pspolygon(-1,0)(-1,0.707)(-0.667,0.707)(-0.667,0)
\pspolygon(1,0)(1,0.707)(1.333,0.707)(1.333,1)(1.5,1)(1.5,0.707)(2,0.707)(2,0)(1,0)
\uput[-90](2.47,-2.4pt){$\xi$}
\uput[-90](1.5,-2.4pt){$\frac{3}{2}$}
\uput[-90](1.333,-2.4pt){$\frac{4}{3}$}
\uput[-90](-1.5,-2.4pt){$-\frac{3}{2}$}
\uput[-90](-.667,-2.4pt){$-\frac{2}{3}$}
 \psline[linewidth=.5pt](1.5,0)(1.5,-0.03)
 \psline[linewidth=.5pt](1.333,0)(1.333,-0.03)
 \psline[linewidth=.5pt](-1.5,0)(-1.5,-0.03)
\psline[linewidth=.5pt](-.667,0)(-0.667,-0.03)
\uput[180](-0.04,-0.707){$-\frac{1}{\sqrt{2}}$}
\uput[180](-2.4pt,0.707){$\frac{1}{\sqrt{2}}$}
 \psline[linewidth=.5pt](-0.03,0.707)(0.03,.707)
 \psline[linewidth=.5pt](-0.03,-0.707)(0.03,-.707)
\end{pspicture}
\centering
      \caption{Graph of $\hat\psi$.}
      \label{fig:ccmw_example}
\end{figure}

We will first show that  $\af{\psi,3/2,\Z}$ indeed is a Parseval frame. By \cite[Corollary 2]{MR1793418} an affine system of the form $\af{\lambda^{-1/2}\psi,3/2,1/\lambda\,\Z}$ is a Parseval frame if, and only if, for almost every~$\xi \in \R$, 
   \begin{align}
    \label{eq:diagonalterm}
    \sum_{j\in \Z} \abs{\hat \psi((\tfrac{3}{2})^j \xi)}^2 &=1 \\
\label{eq:nondiagonalterm}
    \sum_{j=0}^s\hat{\psi}((\tfrac{3}{2})^j\xi) \overline{\hat 
      \psi ((\tfrac{3}{2})^j (\xi+2^s \lambda t))} &= 0 \quad \text{for } s = 0,1,\dots \text{ and }  t \in \Z \setminus (2\Z \union 3\Z).
  \end{align}
 It is easy to see, \eg from Figure~\ref{fig:ccmw_example}, that equation~\eqref{eq:diagonalterm} is satisfied. Since 
 \begin{equation}
\supp \hat\psi ((\tfrac{3}{2})^j (\cdot+2^s t)) \subset \itvcc{-2}{2}-2^st \quad \text{for $j \ge 0$},\label{eq:8}
\end{equation}
we have that
\[ \meas{\supp \hat{\psi}((\tfrac{3}{2})^j\cdot) \cap \supp\hat 
      \psi ((\tfrac{3}{2})^j (\cdot+2^s  t)) } =0 \quad \text{for }\abs{t}\ge 5 \text{ and } s\ge 0.\]
Therefore, we only need to verify \eqref{eq:nondiagonalterm} with
$\lambda =1$ for $t=\pm 1$. This is trivial when $s=0$ since $\hat\psi$ has disjoint support with both $\hat\psi(\cdot{} +1)$ and $\hat\psi(\cdot{} -1)$. For $t = \pm 1$ we have
\begin{align*}
  \hat\psi(\xi) \hat\psi(\xi-2t) +   \hat\psi(\tfrac{3}{2}\xi) \hat\psi(\tfrac{3}{2} (\xi-2t)) = \tfrac{1}{2} \charfct{\itvcc{t}{1/3+t}}(\xi) - \tfrac{1}{2} \charfct{\itvcc{t}{1/3+t}}(\xi) =0, 
\end{align*}
as seen from Figure~\ref{fig:ccmw_example} and \ref{fig:ccmw_example-dilated}. This shows \eqref{eq:nondiagonalterm} for $s=1$. When $s\ge 2$ the equations in \eqref{eq:nondiagonalterm} are trivially satisfied by \eqref{eq:8}, which, in turn, proves that $\af{\psi,3/2,\Z}$ is a Parseval frame. 
\begin{figure}[htbp]
\psset{xunit=2cm,yunit=2.cm}
\begin{pspicture}(-3.5,-1.)(4,1.2)
\psaxes[Dx=1, Dy=1,tickstyle=bottom,linewidth=1.2pt]{->}(0,0)(-3.5,-1)(4,1.2) 
 \psline[linewidth=.5pt](-0.03,0.707)(0.03,.707)
 \psline[linewidth=.5pt](-0.03,-0.707)(0.03,-.707)
\pspolygon(-1,0)(-1,-0.707)(-0.667,-0.707)(-.667,0)
\pspolygon(-0.667,0)(-0.667,0.707)(-0.444,0.707)(-0.444,0)
\pspolygon(0.667,0)(0.667,0.707)(0.888,0.707)(.888,1)(1.,1)(1.,0.707)(1.333,0.707)(1.333,0)(0.667,0)
\psset{linestyle=dashed, linewidth=1.2pt}
\pspolygon(-3,0)(-3,-0.707)(-2.667,-0.707)(-2.667,0)
\pspolygon(-2.667,0)(-2.667,0.707)(-2.444,0.707)(-2.444,0)
\pspolygon(-1.333,0)(-1.333,0.707)(-1.112,0.707)(-1.112,1)(-1.,1)(-1.,0.707)(-.667,0.707)(-.667,0)(-1.333,0)
\psset{linestyle=dotted, linewidth=1.2pt}
\pspolygon(1,0)(1,-0.707)(1.333,-0.707)(1.333,0)
\pspolygon(1.333,0)(1.333,0.707)(1.566,0.707)(1.566,0)
\pspolygon(2.667,0)(2.667,0.707)(2.888,0.707)(2.888,1)(3.,1)(3.,0.707)(3.333,0.707)(3.333,0)(3.667,0)
\uput[-90](3.97,-2.4pt){$\xi$}
\uput[180](-0.04,-0.707){$-\frac{1}{\sqrt{2}}$}
\uput[180](-2.4pt,0.707){$\frac{1}{\sqrt{2}}$}
\end{pspicture}
\centering
      \caption{Graph of  $\hat\psi (\tfrac{3}{2}(\xi +2))$ (dashed),  $\hat\psi (\tfrac{3}{2}\xi)$ (solid), and $\hat\psi (\tfrac{3}{2}(\xi-2))$ (dotted). }
      \label{fig:ccmw_example-dilated}
\end{figure}

For $\lambda =2$ equation~\eqref{eq:nondiagonalterm} with $s=0$ and $t = \pm 1$ becomes
$\hat\psi(\xi)\overline{\hat\psi(\xi\pm 2)}=0$ which is clearly not
satisfied (see Figure~\ref{fig:ccmw_example}). Therefore
$\af{2^{-1/2}\psi, 3/2, \Z/2}$ is not a Parseval frame contradicting
\cite[Proposition~1(ii)]{CCMW}. On the other hand, we observe directly
from the characterizing equations \eqref{eq:diagonalterm} and \eqref{eq:nondiagonalterm} that $\af{\lambda^{-1/2}\psi,3/2,1/\lambda\Z}$ actually is a Parseval frame for any $\lambda \ge 4$.

The proof of \cite[Proposition~1(ii)]{CCMW} is based on higher
dimensional analogues of the characterizing equations
 \eqref{eq:diagonalterm} and \eqref{eq:nondiagonalterm}. The mistake in the proof follows from the
fact that the parameter $s$ is not only present in the $\hat\psi
((\tfrac{3}{2})^j (\cdot+2^s t))$-term, but also determines the number
of terms in the sum \eqref{eq:nondiagonalterm}. Therefore one cannot
replace only one instance of $s$ with $s+1$ as done in the proof without changing the conditions in a profound way.  
\end{example}

\subsection{Oversampling of dual frames with the support condition}
\label{sec:first-overs-dual}

The following result is an analogue of Theorem
\ref{thm:oversamp-frames-support} for dual affine frames 
using the
weaker condition \eqref{eq:condition-J-weak} instead of \eqref{eq:condition-J-strong}.

\begin{theorem}\label{thm:supp-cond-oversamp-dual-frames}
  Let $A \in GL_n(\R)$ be expansive and $\Psi, \Phi \subset L^2(\Rn)$.
  For $J_0 \in \N_0$ take
  $\LL \supset \Zn$ to be a lattice in $\Rn$ satisfying \eqref{eq:condition-J-weak}.
Suppose that, for every $l=1,\dots, L$, 
\begin{equation}
  \label{eq:condition-supp-dual}
  \hat \psi_l(\xi) \hat \phi_l (\xi + k)=0 \qquad \text{for all $k \in \Zn
    \setminus B^{J_0} \Zn$,}
\end{equation}
and that $\af{d(\LL)^{1/2} \Psi,A,\LL}$ and $\af{d(\LL)^{1/2}
  \Phi,A,\LL}$ are Bessel sequences.
 If $\af{\Psi}$ and $\af{\Phi}$ are dual frames, then so are
$\af{d(\LL)^{1/2} \Psi,A,\LL}$ and $\af{d(\LL)^{1/2} \Phi,A,\LL}$. 
\end{theorem}

\begin{proof}
  By \cite[Theorem 4]{CCMW} it suffices to prove that
  \[t^\LL_\alpha(\xi) = \sum_{l=1}^L
  \sum_{\sumstyle{j\in\Z}{B^{-j}\alpha \in \LL^\ast}} \hat
  \psi_l(B^{-j}\xi) \overline{\hat \phi_l(B^{-j}(\xi+\alpha))}= \delta_{\alpha,0}\] for $\alpha \in \Zn$. From
  our hypothesis we have that $t^{\Zn}_\alpha(\xi) = \delta_{\alpha,0}$.
  Fix $\alpha \in \Zn$. We can assume that
  $\setprop{j\in\Z}{B^{-j}\alpha \in \LL^\ast}$ is non-empty;
  otherwise, we have nothing to prove. In this case, we claim that
\[ \setprop{j\in\Z}{B^{-j}\alpha \in \LL^\ast} \supset
\setprop{j\in\Z}{B^{-j}\alpha \in B^{J_0} \Zn}.\]
To see this take $j \in \Z$ so that $B^{-j}\alpha \in B^{J_0}\Zn$.
By the assumption on $\alpha$, there is a $j_0 \in \Z$ such that $B^{-j_0}\alpha \in
\LL^\ast$. Thus, by \eqref{eq:condition-J-weak},
\begin{align*}
  \alpha \in B^{J_0+j} \Zn \cap B^{j_0}\LL^\ast = B^{j_0}(B^{J_0+j-j_0} \Zn \cap
  \LL^\ast) \subset B^{j_0}(B^{j-j_0}  \LL^\ast) = B^j \LL^\ast,
 \end{align*}
that is, $B^{-j}\alpha \in \LL^\ast$. 

In $t^{\Zn}_\alpha(\xi)$ we sum over $\setprop{j\in\Z}{B^{-j}\alpha
  \in \Zn}$, but since $\hat \psi_l(B^{-j}\xi) \hat \phi_l (B^{-j}(\xi +
\alpha))=0$ for all $B^{-j} \alpha \in \Zn \setminus B^{J_0} \Zn$ by
\eqref{eq:condition-supp-dual}, this can be replaced with
$\setprop{j\in\Z}{B^{-j}\alpha \in B^{J_0} \Zn}$. In other words,
$t^{\Zn}_\alpha(\xi) = t^{A^{-j} \Zn}_\alpha(\xi)$. Therefore, by
\eqref{eq:6} and the claim, we conclude that $t^\LL_\alpha(\xi) =
\delta_{\alpha,0}$.
\end{proof}

\section{Tight oversampling and shift invariance gain}\label{sec:gain}

In the final section we restrict our attention to the setting of integer dilations. Our goal is to provide a link between the improved shift invariance of the core space of an orthogonal wavelet and dilation matrix oversampling, \ie oversampling by special classes of lattices of the form $\LL=A^{-s}\Zn$, $s\in \N$. 
These types of lattices are the antipodes of admissible lattices for oversampling listed in Proposition \ref{thm:assump-on-LL}. Thus, they might appear as the worst choice of lattices for showing oversampling results. 

In spite of this, we show that this class of lattices plays an
important role in linking oversampling with additional shift
invariance of the core space. More precisely, the preservation of
the tight frame property when oversampling by such lattices is actually
equivalent with the membership in Behera--Weber classes of wavelets
\cite{Be, We}. Other results on the dilation matrix oversampling
were obtained earlier in \cite{BW03, CCMW}. By Remark \ref{r2}, our
results on oversampling with the support condition, Theorems
\ref{thm:oversamp-frames-support} and
\ref{thm:supp-cond-oversamp-dual-frames}, are also applicable for
lattices $\LL=A^{-s}\bzn$ since both \eqref{eq:condition-J-strong} and \eqref{eq:condition-J-weak}  hold with $J_0=s$.

Suppose that $\psi \in L^2(\brn)$ is a semi-orthogonal Parseval wavelet. This means that $\af{\psi} = \setprop{\psi_{j,k}}{j\in\Z,\ k \in\Zn}$ is a Parseval frame and 
\[
\lan \psi_{j,k}, \psi_{j',k'} \ran =0 \qquad\text{for all }
j\not= j' \in\Z,\ k,k'\in\Zn.
\]
The {\it space of negative dilates} of $\psi$ is defined by 
\begin{equation*} 
V=V(\psi)=\overline{\spa}\{\psi_{j,k} : j<0,k\in\bzn \}.
\end{equation*}

Following Behera \cite{Be} and Weber \cite{We} we define the classes of wavelets with respect to the extent of shift invariance of corresponding spaces of negative dilates.

\begin{definition}
We say that a semi-orthogonal Parseval wavelet $\psi$ belongs to the class $\mathcal L_r$, $r\in \bn \cup \{0\}$, if $V(\Psi)$ is $A^{-r}\bzn$-SI. We say that $\psi \in \mathcal L_\infty$ if $V(\Psi)$ is invariant under all translations $T_y$, $y\in\brn$.
\end{definition}

By definition Behera--Weber classes are nested, \ie $\mathcal L_r \subset \mathcal L_{r+1}$. However, it is much less obvious that the above inclusions are proper, \ie $\mathcal L_r \not= \mathcal L_{r+1}$ for all $r=0,1,\ldots$. This result is due to Behera \cite[Theorem 3.4]{Be}.

\begin{theorem}\label{bh}
Suppose $\psi \in L^2(\brn)$ is a semi-orthogonal Parseval wavelet associated with dilation $A\in M_n(\bz)$. Then, for any $r\in\bn \cup\{\infty\}$, the following are equivalent:
\begin{enumerate}[(i)]
\item $\psi \in \mathcal L_r$,
\item $W(\psi)=\ov{\spa}\{\psi_{0,k} : k\in\bzn \}$ is $A^{-r}\Zn$-SI,
\item $|K \cap (K+k)| =0$ for all $k\in \bzn \setminus B^r\bzn$, where $K=\supp \hat\psi$,
\item The oversampled wavelet systems
$\af{|\det A|^{-s/2}\psi,A,A^{-s}\bzn}$ are Parseval frames for each $s=1,\ldots,r$.
\end{enumerate}
\end{theorem}

When $r=\infty$ we use the convention that $A^{-r}\Zn=\Rn$ and $B^r\Zn=\{0\}$. Thus, (ii) reads that $W(\psi)$ is invariant under all translations in $\Rn$, which is easily seen to be equivalent with $\psi \in \mathcal L_\infty$, and thus $\psi$ is an MSF semi-orthogonal Parseval wavelet, that is $|\hat\psi|= \ch{K}$ for some measurable set $K \subset \Rn$.

\begin{proof}
The equivalence (i) $\iff$ (ii) for orthogonal wavelets is shown in \cite[Lemma 2.2]{Be}. The following argument extends this result to Parseval semi-orthogonal wavelets. 

First, assume (ii).
If $W(\psi)$ is $A^{-r}\Zn$-SI, then so are the spaces $ D_{A^j}(W(\psi))$ for $j\ge 0$.
Since $\psi$ is a semi-orthogonal wavelet, we have
\[
V(\psi) = \bigoplus_{j< 0} D_{A^j}(W(\psi)) = \bigg(\bigoplus_{j\ge 0} D_{A^j}(W(\psi)) \bigg)^\perp,
\]
and thus (i) holds. Conversely, (i) and the identity $W(\psi) =D_{A}(V(\psi)) \ominus V(\psi)$ imply (ii).

The equivalence (ii) $\iff$ (iii) for orthonormal wavelets can be
found in \cite[Theorem 2.5]{Be}. The following argument extends this
result to Parseval semi-orthogonal wavelets as a consequence of a more general lemma about shift-invariance gain for SI spaces, which is motivated by the results from \cite{achkm}. Indeed, (ii) $\iff$ (iii) follows from Lemma \ref{gain-si} with $V=W(\psi)$ and $\LL=A^{-r}\bzn$.

\begin{lemma}\label{gain-si}
Suppose that $V$ is a principal $\bzn$-SI subspace generated by $\vp$, and $\bzn \subset \LL$. Then, $V$ is $\LL$-SI if and only if
\begin{equation}\label{gain-si0}
|K \cap (k+K)| = 0
\qquad\text{for all } k\in \bzn \setminus \LL^*,
\text{ where }K=\supp \hat\vp.
\end{equation}
\end{lemma}

\begin{proof}
The fact that $V$ is a principal SI space with respect to shifts in $\Zn$ implies that
\[
V = \{ f\in L^2(\Rn): \hat f(\xi) = m(\xi) \hat\vp(\xi) 
\quad\text{a.e. $\xi$, for some measurable $\Zn$-periodic } m\},
\]
see \cite{BDR} or \cite{B00}.
If $V$ is also $\LL$-SI, then we also have
\[
V = \{ f\in L^2(\Rn): \hat f(\xi) = \tilde m(\xi) \hat\vp(\xi) 
\quad\text{a.e. $\xi$, for some measurable $\LL^*$-periodic } \tilde m\}.
\]
Let $S= \bigcup_{\gamma \in\LL^*} ([0,1]^n + \gamma)$ and $K=\supp \hat\vp$.  Since $\tilde m=\ch{S}$  is $\LL^*$-periodic, there exists $\Zn$-periodic $m$ such that
\[
\tilde m(\xi) \hat\vp(\xi) = m(\xi) \hat\vp(\xi) \qquad\text{for a.e. }\xi.
\]
This implies that $m(\xi)=1$ for a.e. $\xi \in K \cap S$. Thus, for a fixed $k_0\in \bzn \setminus \LL^*$, we have $\tilde m(\xi-k_0)=0$ and $m(\xi-k_0)=1$ for a.e. $\xi \in K\cap S$. Thus, we must have $\hat\vp(\xi-k_0)=0$ for such $\xi$, which shows that $|K \cap (k_0+K)\cap S|=0$. Replacing $S$ by its translate $k+S$, shows that $|K \cap (K+k_0)\cap (k+S)|=0$ for all $k\in\Zn$, and thus \eqref{gain-si0} holds. 

Conversely, assume \eqref{gain-si0}. Let $\mathcal D \ni 0$ be the set
of representatives of distinct cosets of $\bzn/ \LL^*$. Suppose that
$f\in L^2(\Rn)$ belongs to the $\LL$-SI space generated by $\vp$, \ie $\hat f(\xi) = \tilde m(\xi) \hat\vp(\xi)$
a.e. for some measurable $\LL^*$-periodic function $\tilde m$. Define the $\bzn$-periodic function 
\[
m(\xi) = \sum_{d\in \mathcal D} \tilde m(\xi+d) \ch{\tilde K}(\xi+d),\qquad\text{where } \tilde K = \bigcup_{\ga \in \LL^*}(K+\gamma).
\]
Our assumption \eqref{gain-si0} implies that $|(\tilde K -d) \cap K|=0$ for $d\in \mathcal D \setminus \{0\}$, and hence
\[
m(\xi) \hat\vp(\xi) = \sum_{d\in \mathcal D} \tilde m(\xi+d) \ch{\tilde K}(\xi+d) \hat\vp(\xi) = \tilde m(\xi) \hat\vp(\xi)
\qquad\text{for a.e. }\xi.
\]
This shows that $\LL$-SI space generated by $\vp$ actually coincides with $V$. This completes the proof of Lemma \ref{gain-si}.
\end{proof}

Finally, the equivalence (iii) $\iff$ (iv) can be deduced from wavelet
characterizing equations for integer dilations \cite{B01, Cal} as in
the work of Catal\'an \cite{Cat} and Chui and Sun \cite{CSun1, CSun2}.
Indeed, the fact that both $\af{\psi,A,\bzn}$ and $\af{|\det
  A|^{-s/2}\psi,A,A^{-s}\bzn}$, $s\in \N$, are Parseval frames implies that
\begin{equation}\label{gain-si2}
\sum_{j=0}^{s-1} \hat\psi(B^j\xi) \ov{\hat\psi(B^j(\xi+q))}=0
\qquad\text{for a.e. $\xi$ and for }q\in \Zn \setminus B\Zn.
\end{equation}
In particular, \eqref{gain-si2} with $s=1$ implies that $|K\cap (q+K)|=0$ for $q\in \Zn \setminus B\Zn$. Then, by induction \eqref{gain-si2} implies that $|K\cap (B^{s-1}q+K)|=0$ for $q\in \Zn \setminus B\Zn$ and $s=1,\ldots,r$. Since, $\bigcup_{s=1}^r B^{s-1}(\Zn \setminus B\Zn) = \Zn \setminus B^r\Zn$, we have (iii). Conversely, (iii) implies that formula \eqref{gain-si2} holds for $s=1,\ldots,r$. 
Thus, by \cite[Theorem 2.1]{CSun2} the affine systems $\af{|\det A|^{-s/2}\psi,A,A^{-s}\bzn}$, $s=1,\ldots,r$, are Parseval frames, which is the assertion (iv). This completes the proof of Theorem \ref{bh}.
\end{proof}

\bibliographystyle{plainnat}

\end{document}